\documentclass[10pt, reqno]{amsart}
\usepackage{graphicx, amssymb, amsmath, amsthm}
\usepackage{tikz}
\numberwithin{equation}{section}

\usepackage{mathtools,color,txfonts}

\usepackage{amscd}

\newcommand{\R}{\mathbb{R}}

\newtheorem{remark}{Remark}[section]
\newtheorem{thm}{Theorem}[section]
\newtheorem{lemma}[thm]{Lemma}
\newtheorem{corollary}[thm]{Corollary}
\newtheorem{proposition}[thm]{Proposition}

\makeatletter
\newcommand{\Extend}[5]{\ext@arrow0099{\arrowfill@#1#2#3}{#4}{#5}}
\makeatother

\begin{document}

\title[dimension of divergence set of Schr\"{o}dinger  with complex time ]{On the dimension of divergence sets of Schr\"{o}dinger equation with complex time}

\author{Jiye Yuan}
\address{The Graduate School of China Academy of Engineering Physics, P. O. Box 2101, Beijing, China, 100088}
\email{yuan\_jiye@126.com}

\author{Tengfei Zhao}%
\address{Beijing Computational Science Research Center,
No. 10 West Dongbeiwang Road, Haidian District, Beijing, China, 100193 }
\email{zhao\_tengfei@csrc.ac.cn}%

\author{Jiqiang Zheng}
\address{Institute of Applied Physics and Computational Mathematics,
P. O. Box 8009,\ Beijing,\ China,\ 100088}
\email{zheng\_jiqiang@iapcm.ac.cn, zhengjiqiang@gmail.com}

\begin{abstract}
This article studies the pointwise convergence for the fractional
Schr\"odinger operator $P^{t}_{a,\gamma}$ with complex time in one spatial dimension.
Through establishing $L^2$-maximal estimates for
initial datum in $H^{s}(\mathbb{R})$,
we see that the solution converges to the initial
data almost everywhere  with $s>\frac14 a(1-\frac1\gamma)_+$ when $0<a<1$  and $s>\frac{1}{2}(1-\frac{1}{\gamma})_{+}$ when $a=1$.
By constructing counterexamples,
we show that this result is  almost
 sharp up to the endpoint.
These results extends the results of P. Sj\"olin, F. Soria and A. Baily.
Second, we  study the Hausdorff dimension of the set of the divergent points,
by showing some $L^1$-maximal estimates with respect to  general Borel measure.
Our results reflect the interaction between dispersion effect and dissipation effect, arising from the fractional Schr\"{o}dinger type operator $P^{t}_{a,\gamma}$ with the complex time.
\end{abstract}

 \maketitle

\begin{center}
 \begin{minipage}{100mm}
   { \small {{\bf Key Words:}  Ginzburg--Landau equation;  maximal inequality estimate; pointwise convergence; Hausdorff dimension.
   }
      {}
   }\\
    { \small {\bf AMS Classification:}
      {42B25,  35Q56, 47A63.}
      }
 \end{minipage}
 \end{center}


\section{Introduction}

\noindent

In this article, we are going to study pointwise
convergence of the fractional Schr\"{o}dinger  operator with complex time in one spatial dimension as follows
\begin{equation}\label{srd-complex}
P^{t}_{a,\gamma}f(x) = e^{ig(t)(-\Delta)^{\frac{a}{2}}}f(x)
= \int_{\mathbb{R}}\hat{f}(\xi)e^{it|\xi|^{a} -t^{\gamma}|\xi|^{a}}e^{ix\xi}\,\mathrm{}{d}\xi,
\end{equation}
where $t>0$, $\gamma >0$, $g(t)=t+it^{\gamma}$, and $a>0$.
The operator $P^{t}_{a,\gamma}$ generates several classical equations, for example,
\begin{enumerate}

  \item If  $g(t)=it$, \eqref{srd-complex} is the solution to
  the linear fractional dissipative equation (see \cite{MYZ,ouhabaz})
  \begin{equation}\label{bas-disp}
 \begin{cases}
 \partial_{t}u  +(-\Delta)^\frac{a}{2} u = 0, ~~(t,x) \in [0,\infty) \times \R\\
    u(x,0) = f(x),~~x\in \R.
  \end{cases}
\end{equation}

   \item If $g(t)=t$ and $a=2$, \eqref{srd-complex} is the solution
    to the basic and universal form of the Schr\"{o}dinger
 equation\footnote{This corresponds to the case of $\gamma =\infty$ when we consider $t\in[0,1)$.},
   \begin{equation}\label{bas-schr}
 \begin{cases}
 i\partial_{t}u  -\Delta u = 0, ~~(t,x) \in \R \times \R\\
    u(x,0) = f(x),~~x\in \R.
  \end{cases}
\end{equation}

\item If $g(t)= e^{i\theta} t$  and $a=2$,  \eqref{srd-complex} is the solution
    to  the linear complex  Ginzburg--Landau equation
  \begin{equation}\label{GL}
 \begin{cases}
 \partial_{t}u  - e^{i\theta }\Delta u = 0, ~~(t,x) \in [0,\infty) \times \R\\
    u(x,0) = f(x),~~x\in \R,
  \end{cases}
\end{equation}
  where $\theta \in[-\frac{\pi}{2},\frac{\pi}{2} ]$,  see \cite{CDW-2013} for example.

\end{enumerate}

For the equation \eqref{bas-disp}, Miao-Yuan-Zhang \cite{MYZ} obtained pointwise estimates  of the kernel function $P(x,t,a)$
of  the semigroup $e^{-t (-\Delta)^{\frac{a}{2}}}$,
\begin{equation}
|P(x,t,a)| \leq C \frac{t}{(t^\frac2a +|x|^2)^{\frac{1+a}{2}}} .
\end{equation}
From this estimate, one can find that $|e^{-t (-\Delta)^{\frac{a}{2}}} f(x)| \leq C  \mathcal{M}(f)(x)$,
where $\mathcal{M}$ is the Hardy-Littlewood maximal operator.
By the boundedness of $\mathcal{M}$, for $f\in L^{p}(\R)$ with $1\le p<\infty$, we have
$$
e^{-t (-\Delta)^{\frac{a}{2}}} f(x) \rightarrow f(x), \quad \text{for} \,\,a.e.\,\,x\in \R,
$$
as the time $t$ tends to $0^+$.

For the equation \eqref{bas-schr}, Carleson \cite{Carleson} put forward a question about the range of exponent $\sigma$ for the Sobolev space $H^{\sigma}(\mathbb{R}^{n})$ such that
for $ f\in H^{\sigma}(\mathbb{R}^{n})$, there is
$$
e^{it\Delta}f(x) \rightarrow f(x) \qquad \text{ a.~e.~} \quad x\in \mathbb{R}^{n},
$$
as the time t tends to $0$.
He proved the almost everywhere  convergence for the exponent $\sigma \geq \frac{1}{4}$ in dimension 1,
which is sharp by the counterexamples given by Dahlberg and Kenig\cite{D-K}.
Barcelo, Bennet, Carbery and Rogers \cite{BBCR} refined these  results by
showing the divergent set such that, for $\frac14 \leq \sigma \leq \frac12 $,
$$
\textup{~ dim~} \big\{x\in\R: e^{it\Delta}f(x) \not\rightarrow f(x), \text{~as~}t\rightarrow0 \big\}  ~\leq ~1-2\sigma ,
$$
where dim $U$ is the Hausdorff dimension of a set $U \subset \R.$

This paper is devoted to the study of the pointwise convergence and the Hausdorff dimension of divergent set of  the operator
$P^{t}_{a,\gamma}$.
For the sake of simplicity, we only recall some related results in one spatial dimension.
For the case $g(t)=t$, Sj\"{o}lin\cite{Sjolin}
proved that the almost everywhere convergence holds for the operator $e^{it(-\Delta)^{\frac{a}{2}}}$ if and only if $\sigma\ge \frac{1}{4}$ when $a>1$.
In \cite{Walther,Walther-global}, Walther considered 
the fractional  Schr\"{o}dinger  operator $e^{it(-\Delta)^{\frac{a}{2}}}$ of a concave phase case,  that is, $a \in(0,1)$, and proved almost sharp $L^2$ maximal estimates(up to the endpoint) for functions in $H^s(\R)$ if $s>\frac{a}{4}$. We also refer to Rogers and Villarroya \cite{RV-2008} for the half wave operator $e^{it\sqrt{-\Delta}}$, where they showed the $L^2$ maximal estimates for $s>\frac12.$
For the   operator $P^{t}_{a,\gamma}$, Sj\"{o}lin in \cite{Sjo} and in \cite{Sjo So} together with Soria
 studied the classical Schr\"{o}dinger operator with a complex parameter($a=2$ and $\gamma\in(0,\infty)$).
Later, using the Kolmogrov-Selierstov-Plessner method, Bailey\cite{Bailey} improved their results 
to the case $a>1$.

We first consider positive parts question in the case $0<a\leq1$. Let $s_{+} = \max \{0, s\}$ for $s\in \mathbb{R}$
and define the  maximal operator $P^{*}_{a,\gamma}$ by
$$
P^{*}_{a,\gamma}f(x) = \sup\limits_{0 < t <1} \big|P^{t}_{a,\gamma}f(x)\big|.
$$
The following is established:

\begin{thm}\label{main-thm-1}
\begin{enumerate}
  \item[(i)]{\bf Local estimate :}
For $a\in(0,1)$ and $\gamma \in(0,\infty)$, we have
\begin{equation}\label{local-maximal-estimate}
\left\|P^{*}_{a,\gamma}f(x)\right\|_{L^{2}(B(0,1))} \lesssim \|f\|_{H^{s}(\mathbb{R})},
\end{equation}
for $ f\in H^{s}(\mathbb{R})$ and $s>\frac{1}{4}a(1-\frac{1}{\gamma})_{+}$. While for $a=1$ and
$\gamma\in(0,\infty)$, \eqref{local-maximal-estimate} holds for $s>\frac{1}{2}(1-\frac{1}{\gamma})_{+}.$

  \item[(ii)]{\bf Global estimate :}
For $a\in(0,1) $ and $\gamma \in(1,\infty)$, we have
\begin{equation}\label{maximal-est-Hs}
\left\|P^{*}_{a,\gamma}f(x)\right\|_{L^{2}(  \R)} \lesssim \|f\|_{H^{s}(\mathbb{R})},
\end{equation}
for $ f\in H^{s}(\mathbb{R})$ and $s > \frac{1}{4}a(1-\frac{1}{\gamma})$.
While for $a=1$ and $\gamma\in(1,\infty)$, \eqref{maximal-est-Hs} holds for $s>\frac{1}{2}(1-\frac1\gamma)$.
  \item[(iii)]{\bf $L^{p}$ estimate :} For $a\in(0,1]$ and $\gamma\in(0,1]$, the maximal estimate holds in $L^p(\R)$ for $1<p\le\infty$, that is,
\begin{equation}\label{maximal-est-L2}
\left\|P^{*}_{a,\gamma}f(x)\right\|_{L^{p}(  \R)} \lesssim \|f\|_{L^{p}(  \R)}, \quad \text{for}\,\,f\in L^p(\mathbb{R}).
\end{equation}
For $p = 1$, we have
\begin{equation}\label{p=1}
\left|\{x\in\R: |P^{\ast}_{a,\gamma}f(x)| > \lambda\}\right| < C \frac{\|f\|_{L^{1}}}{\lambda},
\end{equation}
for $f\in L^{1}(\R)$, $\lambda >0$.

\end{enumerate}
\end{thm}

\begin{remark}
$(1)$  Theorem \ref{main-thm-1} extends the results of
\cite{Bailey,Sjo,Sjo So}  to the case that $0<a\leq1$.

$(2)$ For $\gamma\in(0,1]$,
  we note from Remark \ref{remark-Poisson} below that the dissipative part plays a leading role,
  then we can obtain the boundedness of the operator $P^{t}_{a,\gamma}$,
 in accord with the operator $e^{-t(-\Delta)^{\frac{a}{2}}}$ for any $s>0,$ see the detail in Remark \ref{remark-Poisson}.

$(3)$ For $\gamma > 1$,
   if we just consider the dispersive effect for the operator $e^{it(-\Delta)^\frac{a}{2}}$, we can get
  \eqref{maximal-est-Hs} with $s>\frac{a}{4}$  for $a<1$, see Remark \ref{remark-Poisson} below for more detail.
Furthermore, we consider the dissipative effect for the operator $P^{t}_{a,\gamma}$, we can improve this result  to $s>\frac{a}{4}\big(1-\frac{1}{\gamma}\big)$ for $a<1$ as in Theorem \ref{main-thm-1}.

$(4)$ From Theorem  \ref{main-thm-1} and the result in \cite{Bailey,Sjo,Sjo So}, we know that the index of regularity $s$ is not continuous with respect to $a$ at the point $a=1$ due to the finite speed propagation for the wave equation, see Fig.1 below. This phenomenon corresponds to fractional Schr\"odinger operator $e^{it(-\Delta)^\frac{a}2}$ and half wave operator $e^{it\sqrt{-\Delta}}$ as in \cite{RV-2008,Sjolin}.

\begin{center}
 \begin{tikzpicture}[scale=1]
 \draw[->] (0,0) -- (6,0) node[anchor=north] {$a$};
\draw[->] (0,0) -- (0,3)  node[anchor=east] {$s$};
\path (2,-0.8) node(caption){{\rm Fig 1}. $\gamma>1$};  

\draw (3.9,0.83) node[anchor=west] {A};
\draw (4,1.7) node[anchor=west] {B};

 \draw (0,0) node[anchor=north] {O}
 (4,0) node[anchor=north] {$1$}
 (5,0) node[anchor=north] {$\frac{\gamma}{\gamma-1}$};

\draw  (0, 2.3) node[anchor=east] {$\frac12$}
(0, 1.7) node[anchor=east] {$\frac12(1-\frac{1}{\gamma})$}
 (0.5, 1.2) node[anchor=east] {$\frac14$}
  (0, 0.93) node[anchor=east] {$\frac14(1-\frac{1}{\gamma})$};

 \draw[thick] (5,1.2) -- (6,1.2)
(-0.05, 2.3)-- (0.05, 2.3)
(-0.05, 1.7)-- (0.05, 1.7)  (-0.05, 1.2)-- (0.05, 1.2)   (-0.05, 0.96)-- (0.05, 0.96)
(4,-0.05) -- (4,0.05)  (5,-0.05) -- (5,0.05);  
\draw[dashed,thick] (0,0) -- (5,1.2); 

\draw (4,0.96) circle (0.06);
\draw {(4,1.7) circle (0.06)}[fill=gray!90];

\end{tikzpicture}

\end{center}

\end{remark}

The proof of Theorem \ref{main-thm-1} (i,ii) is based on some oscillatory estimates and the Littlewood-Paley decomposition(see Section \ref{osc-est}), which shows the interaction between the dispersive effect and the dissipative effect for the operator $P^{t}_{a,\gamma}$.
Theorem \ref{main-thm-1} (iii) is proved by showing $P^*_{a,1}$ and the fractional
  dissipation $e^{-t(-\Delta)^\frac{a}{2}}$ is bounded by the Hardy-Littlewood maximal functions. 
As a direct consequence, through a standard argument,
 we obtain the almost everywhere convergence results:

\begin{corollary}\label{point}
Let $0<a\le1$. For $f \in L^p(\mathbb{R})$ with $1\le p <\infty$ if $\gamma\in (0,1]$, and $f \in H^{s}(\mathbb{R})$ if $\gamma\in (1,\infty)$, $s > \frac{1}{4}a(1-\frac{1}{\gamma})$ with $0<a<1$, and $s> \frac{1}{2}(1-\frac1\gamma)$ with $a=1$,
 then%
\begin{equation}\label{aeconv}
 \lim_{t\rightarrow 0^+ } P^{t}_{a,\gamma}f(x) =f(x)
\end{equation}
 holds almost everywhere.
\end{corollary}

By employing the theorems of Nikishin\cite{Ninkishin}, we  construct a counterexample to obtain a necessary condition for the pointwise convergence of the operator $P^t_{a,\gamma}$,
which indicates that Corollary \ref{point} is sharp up to the endpoint.
\begin{thm}\label{neg-thm}
For $0<a<1$, if $\gamma >1$, and $s<\frac{1}{4}a(1-\frac{1}{\gamma})$, and for $a=1$, if $\gamma>1$ and $s<\frac{1}{2}(1-\frac{1}{\gamma})$, then
 the almost everywhere convergence \eqref{aeconv} fails.

\end{thm}

Next, we consider the maximal estimate for the operator $P^{t}_{a,\gamma}$ in the cases of general Borel  measures, in order to estimate the Hausdorff dimension of the divergent set. 
Let $\mu$ be a Borel measure on $\R$. The support of a measure $\mu$  is the smallest closed set $F$ such that $\mu(\R\backslash F) = 0$, which is denoted by spt $\mu$. Let $X\subset \mathbb{R}$, then we denote $\mathfrak{M}(X)$ by the set of all Borel measures $\mu$ on $\R$ with $0 < \mu(X) <\infty$ and with compact spt $\mu\subset X$.
And for $s > 0$, the $s$-energy of a Borel measure $\mu$ is denoted by $I_{s}(\mu)$(see \cite{Mattila}), that is
$$
I_{s}(\mu) = \iint_{\mathbb{R} \times \mathbb{R} }|x-y|^{-s}\,\mathrm{d} \mu (x)\,\mathrm{d} \mu (y) .
$$
With these notations, we will prove the following results:
\begin{thm}\label{lemma:main-a}
Let $a\in (0,\infty)$, $\gamma \in (0,\infty)$, and $\mu\in \mathfrak{M}(B(0,1))$.
For $f\in H^{s}(\mathbb{R})$, we have
\begin{enumerate}
\item[(i)]For $a>0$ and $a\neq1$, if $\gamma \in(0,1]$ and  $s\in(0,\frac12)$ or
$\gamma>1$ and  $s\in [\frac14,\frac12)$, then
\begin{equation}\label{energy-est}
\int \sup_{0<t<1}|P^{t}_{a,\gamma}f(x)|\,\mathrm{d} \mu \lesssim
I_{1-2s}(\mu)^\frac12
\|f\|_{H^{s}(\mathbb{R})},
\end{equation}
provided that $I_{1-2s}(\mu)<\infty$.

\item[(ii)] For $0< a<1$, if $\gamma >1$
and $s\in (\frac{1}{4}a(1-\frac1\gamma),\frac14)$, we
have
\begin{equation}\label{L1-sigma}
\int \sup_{0<t<1}|P^{t}_{a,\gamma}f(x)|\,\mathrm{d} \mu \lesssim
I_{\sigma}(\mu)^\frac12
\|f\|_{H^{s}(\mathbb{R})},
\end{equation}
provided that $I_{\sigma}(\mu)<\infty$
with  $\sigma 
=1-2s+\frac{a(4s-1)(\gamma-1)}{2[(a-1)\gamma-a]}.$

\item[(iii)] For $a>1$, if $\gamma\in(1,\frac{a}{a-1})$
and $s\in (\frac{1}{4}a(1-\frac1\gamma),\frac14)$, then there holds \eqref{L1-sigma}.

 \item[(iv)] For $a=1$, if $\gamma>0$ and $s\in(\frac{1}{2}(1-\frac{1}{\gamma})_{+},\frac{1}{2})$, let $\varrho(\gamma,s) = \max\{1-2s,\gamma(1-2s)\}$, then
\begin{equation}\label{energy-est-2}
\int \sup_{0<t<1}|P^{t}_{a,\gamma}f(x)|\,\mathrm{d} \mu \lesssim
I_{\varrho(\gamma,s)}(\mu)^\frac12
\|f\|_{H^{s}(\mathbb{R})},
\end{equation}
provided that $I_{\varrho(\gamma,s)}(\mu)<\infty$.

\end{enumerate}

\end{thm}

\begin{remark}
Especially, let $\mu = \chi_{B(0,1)}\mathcal{L}$, where $\chi_{B(0,1)}$ is the characteristic function on $B(0,1)$ and $\mathcal{L}$ is the Lebesgue measure on $\R$, then we can establish the $L^{1}$ local maximal estimate for the operator $P^{t}_{a,\gamma}$ in the Lebesgue measure by Theorem \ref{lemma:main-a}, from which we can also obtain the pointwise convergence result Corollary \ref{point}.
\end{remark}

 As an application of Theorem \ref{lemma:main-a}, we have
the following estimate on the Hausdorff dimension of the divergent set of the operator $P^t_{a,\gamma}$,
in viewing of  the relation between the energy of a Borel measure and the Hausdorff dimension of a set, see Lemma \ref{dim} below.

\begin{thm}\label{main-thm-dim-1}
Let $a\in (0,\infty)$  and  $\gamma \in (0,\infty)$.
We have following estimates:

\begin{enumerate}
  \item[(i)]  For $a>0$ and $a\neq1$, if $\gamma \in(0,1]$ and  $s\in(0,\frac12]$ or
$\gamma>1$ and  $s\in [\frac14,\frac12]$,  then
\begin{equation}\label{HD-1-1}
\textup{~dim~}   \big\{x\in \mathbb{R}: \lim_{t\rightarrow 0^+} P^{t}_{a,\gamma}f(x) \neq f(x) \big\} \leq 1-2s, 
\end{equation}
for any $f\in H^{s}(\mathbb{R})$.
  \item[(ii)] For $0<a<1$, if  $\gamma >1$ and $s\in (\frac{1}{4}a(1-\frac1\gamma),\frac14)$, then
\begin{equation}\label{HD-1-2}
  \textup{~dim~}  \big\{x\in \mathbb{R}: \lim_{t\rightarrow 0^+} P^{t}_{a,\gamma}f(x) \neq f(x)\big\} \leq  \sigma,  
\end{equation}
for any $f\in H^{s}(\mathbb{R})$ with 
 $\sigma 
=1-2s+\frac{a(4s-1)(\gamma-1)}{2[(a-1)\gamma-a]}.$
 \item[(iii)] For $a>1$, if $\gamma \in(1,\frac{a}{a-1})$ and $s\in (\frac{1}{4}a(1-\frac1\gamma),\frac14)$, then  \eqref{HD-1-2} holds.

 \item[(iv)] For $a=1$, if $\gamma>0$ and $s\in(\frac{1}{2}(1-\frac{1}{\gamma})_{+},\frac{1}{2}]$, let $\varrho(\gamma,s) = \max\{1-2s,\gamma(1-2s)\}$, then
\begin{equation}\label{HD-1-1-2}
\textup{~dim~}   \big\{x\in \mathbb{R}: \lim_{t\rightarrow 0^+} P^{t}_{a,\gamma}f(x) \neq f(x) \big\} \leq \varrho(\gamma,s), 
\end{equation}
for any $f\in H^{s}(\mathbb{R})$.

\end{enumerate}

\end{thm}

\begin{remark}

 In the general Borel measure setting, the authors of \cite{BBCR} and \cite{Mattila} considered the Hausdorff dimension of the divergent set associated with the operator $e^{it(-\Delta)^{\frac{a}{2}}}$ for $a>1$. We first consider this problem for the operator $P^{t}_{a,\gamma}$ for $a>0$ with complex time, which is associated with the linear complex Ginzburg-Landau equation \eqref{GL} for $\gamma = 1$, $a=2$ and $\theta = \frac{\pi}{4}$.

%

\end{remark}

\vskip0.2cm

\vskip0.2cm

This paper is organized as follows.
In the section 2, we will give the main oscillatory estimates associated with the operator $P^{t}_{a,\gamma}$, which plays an important role in the proof of Theorem \ref{main-thm-1} and the Hausdorff dimension of the divergent set.
Section 3 is aimed at proving Theorem \ref{main-thm-1} and discusses the almost everywhere convergence of the solution to the Schr\"{o}dinger equation.
Section 4 is devoted to the proof of Theorem \ref{neg-thm}.
In the section 5, we will study the Hausdorff dimension of the set of the divergent points.

\subsection{Notations}
Finally, we conclude the introduction by giving some notations which
will be used throughout this paper. If $A$ and $B$ are two positive quantities, we write $A\lesssim B$ when there exists a constant $C > 0$ such that $A\leq C B$, where the constant will be clear from the context. We use $\mathcal{S}(\R)$ denote the Schwartz class of functions on the Euclidean space $\R$.

For $\sigma > 0$,  $H^{\sigma}(\mathbb{R})$  denotes the Sobolev space
$$
\Bigl\{f\in L^{2}(\mathbb{R}) : \int_{\mathbb{R}}(1+|\xi|^{2}|)^{\sigma}|\hat{f}(\xi)|^{2}\,\mathrm{d}\xi<\infty\Bigr\}.
$$

For $s\geq 0$ and a Borel set $U \subset \mathbb{R}$, the Hausdorff measures $\mathcal{H}^{s}$ of $U$ can be defined as
$$
\mathcal{H}^{s}(U)=\lim_{\delta\rightarrow 0}\mathcal{H}^{s}_{\delta}(U),
$$
where, for $0<\delta\leq \infty$,
$$
\mathcal{H}^{s}_{\delta}(U)=\inf \Big\{\sum_{j}d(E_{j})^{s} : U\subset \bigcup_{j}E_{j}, d(E_{j})< \delta\Big\}.
$$
And the Hausdorff dimension of a Borel set $U\subset \mathbb{R}$ is equivalently  defined by
$$
\textup{dim} \,\,U = \inf \big\{s : \mathcal{H}^{s}(U) = 0\big\} = \sup\big\{s : \mathcal{H}^{s}(U) = \infty\big\}.
$$

For $X\subset \R$, we denote $|X|$ or $m(X)$ to be the Lebesgue measure of set $X$.




\section{Preliminaries}\label{osc-est}

In this section, we give some oscillatory estimates for latter use.
First, we will utilize the stationary phase analysis and Fourier localization method to obtain the Carleson type estimates  for $a\in(0,1]$.
We also show a Poisson kernel type estimate associated with the operator  $e^{-(1+i)t(-\Delta)^\frac{a}{2} }$.

\subsection{Elementary oscillation  estimates}

In the case $a\in(0,1]$, we will show the following corresponding estimates:
\begin{lemma}\label{maximal estimate}
For $\gamma \in (0, \infty)$,
suppose $a\in(0,1)$ with
$\tfrac{1}{2}a\big(1-\tfrac{1}{\gamma}\big)_+<\alpha<1$,
 and $a=1$ with $(1-\frac{1}{\gamma})_{+} < \alpha<1$.
 Let $\mu  \in \mathcal{S}(\mathbb{R})$,
  compactly supported, positive, even and real-valued.
 Suppose $\chi$ is a compactly supported function and $\chi(\xi)=1$ for
 $|\xi|\leq 1.$

  Then,

 $(i)$ ${\bf Global~~ estimate:}$
  If $a\in(0,1] $,
   then there exists a function $K(x)\in L^{1}(\mathbb{R}) $
    such that for $\forall t_{1}, t_{2} \in (0,1)$
     and $x\in \mathbb{R}$ and $N\in2^{\mathbb{N}}$, we have
\begin{equation}\label{goal}
\left|\int_{\mathbb{R}}e^{i(t_{1}-t_{2})|\xi|^{a}}e^{-ix\xi}
e^{-(t_{1}^{\gamma}+t_{2}^{\gamma})|\xi|^{a}}
(1+\xi^{2})^{-\frac{\alpha}{2}} (1-\chi(\xi))
\mu\big(\tfrac{\xi}{N}\big)
\,\mathrm{d}\xi\right| \le C K(x).
\end{equation}

$(ii)$ ${\bf Local~~ estimate:}$
 If $a\in(0,1)$, then for $x\in B(0,1)$,
 we can take $K(x)$ as
\begin{equation}\label{equ:kxcho}
  K(x)= ~ \begin{cases}
  ~~|x|^{\alpha-1}\quad\text{if}\quad \gamma\in(0,1),\\
 ~ ~|x|^{\alpha-1} + |x|^{-\sigma}\quad\text{if}\quad \gamma>1,
  \end{cases}
\end{equation}
where
\begin{equation}\label{equ:sigadef}
\sigma = \frac{1}{a-1}\Big(\alpha + \frac{1}{2}(a-2)+\frac{a(\alpha-\frac{1}{2})}{(a-1)\gamma-a}\Big).
\end{equation}
While for $a=1$ and $x\in B(0,1)$, we can take $K(x)\in L^{1}(B(0,1))$ as
\begin{equation}\label{goal1}
K(x) = |x|^{\alpha-1} +|x|^{\gamma(\alpha-1)}.
\end{equation}


\end{lemma}

\begin{remark}\label{maximal estimate-a>1}
This lemma extends the results of Bailey \cite{Bailey} to the case that $a\leq1$.
Here, we focus on the high frequency part, since the maximal estimates of low frequency are quite easy.

$(1)$ For the proof of the  global estimate \eqref{goal}, we  will utilize the Fourier localization methods since the critical point depends on the frequency.

$(2)$ 
We can observe that the estimate \eqref{goal}
 holds for $a>1$  and $x\in B(0,1)$ by choosing
 \begin{equation}\label{equ:kxcho-a>1}
  K(x)=\begin{cases}
  |x|^{\alpha-1}\quad\text{if}\quad \gamma\in(0,1), \text{~or~}  \gamma \in (1, \infty) \text{~ and~} \alpha \in [\frac12,1) ,\\
 |x|^{-\sigma}\quad\text{if}\quad\gamma\in(1,\frac{a}{a-1}) \text{~ and~}
 \alpha \in(\frac12a(1-\gamma),\frac12),
  \end{cases}
\end{equation}
where $\sigma$ is as in \eqref{equ:sigadef}. This fact can be deduced from Lemma 2.1 \cite{Bailey}.
%


$(3)$
The local estimates \eqref{equ:kxcho} and \eqref{equ:kxcho-a>1}
 will play an important role in the proof of the maximal estimate for the operator $P^{t}_{a,\gamma}$ in the general Borel measure $\mu\in \mathfrak{M}(B(0,1))$.

%

\end{remark}

In order to prove Lemma \ref{maximal estimate}, we first recall the Van der corput Lemma.
\begin{lemma}[Van der corput lemma, \cite{Stein}]\label{lem:van der}
Suppose $\phi$ is real-valued and smooth in $(a,b)$, $\psi$ is complex-valued and smooth, and that $|\phi^{(k)}(x)|\geq1$ for all $x\in(a,b)$. Then
\begin{equation}\label{equ:vandc}
  \Big|\int_a^b e^{i\lambda\phi(x)}\psi(x)\;dx\Big|\leq c_k\lambda^{-\frac1k}\Big[|\psi(b)|+\int_a^b|\psi'(x)|\;dx\Big]
\end{equation}
holds when
\begin{enumerate}
  \item $k\geq2$ or
  \item  $k=1$ and $\phi'(x)$ is monotonic.
\end{enumerate}
The bound $c_k$ is independent of $\phi$ and $\lambda$.
\end{lemma}


%

Next, we prove the Poisson-type kernel estimates, which
reflects that the dissipative part of the operator paly an  leading role,
when $\gamma\in(0,1]$. The case $\gamma=1.$ is associated with the complex  Ginzburg--Landau equation \eqref{GL}.
Let $L(x,t,a)$ be the convolution kernel of the operator  $e^{-(1+i)t(-\Delta)^\frac{a}{2} }$.
\begin{lemma}[Poisson-type kernel estimates]\label{lemma-Poisson}

For $a>0$, then we have for $x\in\R$ and $t>0$ 
\begin{equation}\label{poisson-comp}
|L(x,t,a)| \lesssim \frac{t}{(t^\frac2a +|x|^2)^{\frac{1+a}{2}}}.
\end{equation}

\end{lemma}

\begin{proof}


By scaling, we have
$
L(x,t,a)= L(xt^{-\frac1a }, 1,a) t^{-\frac{1}a}.
$
It is sufficient to prove
\begin{equation}
  \Big|\int_\mathbb{R} e^{-(1+i)| \xi|^a } e^{ix \xi} d\xi \Big|
  \leq C\min\{1, |x|^{-1-a}\}.
\end{equation}
Since the finiteness is trivial,  we just consider the case $|x|\gg1$.
By integration by parts, we have
\begin{align*}
  \int_\mathbb{R} e^{-(1+i)| \xi|^a } e^{ix \xi} d\xi
& =-\frac{a}{|x|} \int_\mathbb{R} e^{-(1+i)| \xi|^a } e^{ix \xi}  |\xi|^{a-2}\xi d\xi \\
&\lesssim \frac{1}{|x|}  \int_{|\xi| \leq 1/|x| }  |\xi|^{a-1} d\xi
+ \frac{1}{|x|} \Big|\int_{|\xi| \geq 1/|x| }  e^{-(1+i)| \xi|^a } e^{ix \xi}  |\xi|^{a-2}\xi d\xi\Big|\\
&:= C |x|^{-1-a} + |x|^{-1}\Sigma.
\end{align*}
By changing of variables, we have
\begin{align*}
\Sigma
\lesssim &  |x|^{-a} \Big| \int_1^\infty  h(\eta) e^{i\phi(\eta)} d\eta\Big| + |x|^{-a} \Big| \int_\infty^{-1}  h(\eta) e^{i\phi(\eta)} d\eta\Big|,
\end{align*}
where $h(\eta)=e^{-|x|^{-a} \eta}$ and $ \phi(\eta)=\eta^\frac1a-|x|^{-a}\eta.$
By symmetry, it suffices to prove
\begin{equation}\label{h-phi}
\Big|\int_1^\infty  h(\eta) e^{i\phi(\eta)} d\eta\Big| \leq C,\quad \text{ for }|x|\gg1,
\end{equation}
which is a direct consequence of
the basic inequalities
\begin{equation}\label{bd-exp}
|\partial_\xi \big ( e^{-\epsilon|\xi|^{a}} \big) |\lesssim\frac{1}{|\xi|}, \qquad    |\partial_\xi^2 \big ( e^{-\epsilon|\xi|^{a}} \big) | \lesssim \frac{1}{|\xi|^{2}},
\text{ for each }  \xi\neq 0 \text{ and }  \epsilon>0,
\end{equation}
and
 integration by parts.


\end{proof}

\begin{remark}\label{remark-Poisson}
As a direct consequence of the estimate \eqref{poisson-comp}, we have
 for $x\in \mathbb{R}$ and $t>0$,
\begin{equation}\label{maximal}
 |e^{-t(1+i)(-\Delta)^{\frac{a}{2}}}f|\lesssim (\mathcal{M}f)(x).
\end{equation}
On the other hand, by simple modifications of above arguments,
for $a>0$, as stated in the introduction, we have the following estimates\footnote{
This is a well-known result, see \cite{BR-TAMS-1960,MYZ} for example. }
for $P(x,t,a)$ and $\tilde P(x,t,a)$, which are the kernels of the  fraction dissipation operators
$e^{-t(-\Delta)^{\frac{a}{2}}}$ and $t(-\Delta)^{\frac{a}{2}}e^{-t(-\Delta)^{\frac{a}{2}}}$,
\begin{equation}\label{poisson}
|P(x,t,a) | +  |\tilde P(x,t,a) | \lesssim \frac{t}{(t^\frac2a +|x|^2)^{\frac{1+a}{2}}},
\end{equation}
which implies, for $t>0$ and $x\in \R$,
\begin{equation}
 |e^{-t(-\Delta)^\frac{a}{2} } f (x)| + |t(-\Delta)^{\frac{a}{2}} e^{-t(-\Delta)^\frac{a}{2} } f (x)|\lesssim (\mathcal{M}f)(x).
 \end{equation}
Therefore, if $h(t)\leq g(t),$ we have
\begin{equation}
  \Big\|\sup_{0<t<1} |e^{i(t+ig(t))(-\Delta)^\frac{a}{2}  }  f| \Big\|_{L^2(\mathbb{R})}
  \lesssim
\Big\|\sup_{0<t<1} |e^{i(t+ih(t))(-\Delta)^\frac{a}{2} } f| \Big\|_{L^2(\mathbb{R})},
\end{equation}
which extends the results of Lemma $1.4$ in \cite{Bailey}. Especially, one has
\begin{equation}\label{equ:specase}
 \Big\|\sup_{0<t<1} |e^{i(t+it^\gamma)(-\Delta)^\frac{a}{2}  }  f| \Big\|_{L^2(\mathbb{R})}
  \lesssim
  \Big\|\sup_{0<t<1} |e^{it(-\Delta)^\frac{a}{2} } f| \Big\|_{L^2(\mathbb{R})}.
\end{equation}

\end{remark}

\subsection{Proof of the global estimates in Lemma \ref{maximal estimate} }\label{pf-global}

Without loss of generality, we can assume that $t_{2} < t_{1}$. 
Let $$t = t_{1}-t_{2}, \epsilon= t_{1}^{\gamma} + t_{2}^{\gamma}.$$
\eqref{goal} is reduced to show that there exists $K(x)\in L^1(\R)$ such that
\begin{equation}\label{equ:lem2.2red}
  \Big|\int_{\mathbb{R}}e^{it|\xi|^{a}}e^{-ix\xi}
e^{-\epsilon|\xi|^{a}}
(1+\xi^{2})^{-\frac{\alpha}{2}}(1-\chi(\xi))
\mu\big(\tfrac{\xi}{N}\big)
\,\mathrm{d}\xi\Big| \le C K(x).
\end{equation}
To do this, we introduce the dyadic partition of unity\cite{Miao}
\begin{equation}\label{equ:dyddec}
  \chi(\xi) +\sum_{M\geq1}\eta\big(\tfrac{\xi}{M}\big)=1,
\end{equation}
 where $M$ denotes the dyadic integer,
 $\eta(\xi)$  is a smooth function and such that
$
\text{supp~}\,\,\eta(\xi)\subset \big\{\xi\in\R:\;1\leq|\xi|\leq 4\big\},\quad \eta(\mathbb{R})\subset[0,1],
$
and $
\text{supp~}\,\,\chi(\xi)\subset \big[-2,2], \quad \chi(\xi) = 1 \,\,on\,\, \big[-1,1\big].
$

Applying the dyadic partition of unity to \eqref{equ:lem2.2red}, we
estimate
\begin{align*}
  {\rm LHS~ of~} \eqref{equ:lem2.2red}\le \sum_{M\geq1}|\Lambda_M|,
\end{align*}
with
$$
\Lambda_{M}(x) = \int e^{it|\xi|^{a}}e^{-ix\xi} g_M(\xi) d\xi,
$$
and
$$
g_M(\xi)=e^{-\epsilon |\xi|^{a}}(1+|\xi|^{2})^{-\frac{\alpha}{2}}\eta\big(\tfrac{\xi}{M}\big)\mu\big(\tfrac{\xi}{N}\big).
$$
Thus, we further reduce \eqref{equ:lem2.2red} to show that for $\alpha>\tfrac{a}{2}\big(1-\tfrac{1}\gamma\big)_+, $
\begin{equation}\label{L0LN}
\sum_{M\ge1}\|\Lambda_{M}\|_{L^{1}(\mathbb{R})}\leq C<+\infty,
\end{equation}
uniformly for $t,\epsilon\in(0,2)$.

\vspace{-0.15cm}
\subsubsection{\bf Estimation for $\Lambda_{M}$ with $0<a<1$} Recall
\begin{align*}
  \Lambda_M= &  \int e^{it|\xi|^{a}}e^{-ix\xi} e^{-\epsilon |\xi|^{a}}(1+|\xi|^{2})^{-\frac{\alpha}{2}}\eta\big(\tfrac{\xi}{M}\big)\mu\big(\tfrac{\xi}{N}\big) d\xi\\
  \triangleq&M\int e^{i\Phi_{a}(\xi,Mx,M^{a}t)} e^{-\epsilon M^a|\xi|^{a}}(1+|M\xi|^{2})^{-\frac{\alpha}{2}}\eta(\xi)\mu\big(\tfrac{M\xi}{N}\big) d\xi
\end{align*}
with $\Phi_a(\xi,x,t)=t|\xi|^{a}-x\xi.$ A simple computation shows that
$$\partial_\xi\Phi_{a}(\xi,Mx,M^{a}t)=atM^a|\xi|^{a-2}\xi-Mx.$$



We divide the two cases to estimate $\Lambda_{M}$.

{\bf Case 1: $|x|\geq2^{2-a}M^{a-1}t$.} In this region, we have
$$|\partial_{\xi}\Phi_{a}(\xi,Mx,M^{a}t)| > \frac{|Mx|}{2},$$
and
$$
\Big|\bigg(\tfrac{1}{i\Phi_{a}'}\Big(\tfrac{(1+|M\xi|^{2})^{-\frac{\alpha}{2}}\eta e^{-\epsilon|\xi|^{a}}}{i\Phi_{a}'}\Big)'\bigg)'\Big| \lesssim M^{-\alpha}|Mx|^{-2}.
$$
Hence, using  integration by parts twice, we get
\begin{align*}
 |\Lambda_M|  \lesssim M^{1-\alpha}|Mx|^{-2}.
\end{align*}
On the other hand, it is easy to see that
$$|\Lambda_M|\leq M\int(1+|M\xi|^{2})^{-\frac{\alpha}{2}}\eta(\xi) d\xi\leq C M^{1-\alpha}.$$
Hence
\begin{align}\label{equ:lambmlarge}
  \|\Lambda_M\|_{L^1(|x|\geq2^{2-a}M^{a-1}t)}\leq & CM^{1-\alpha}\|(1+|Mx|)^{-2}\|_{L^1(\R)}\leq CM^{-\alpha}.
\end{align}

{\bf Case 2: $|x|\leq2^{2-a}M^{a-1}t$.} In this region,   for $\xi\in [\frac{1}{2},2]$, we have
$$|\partial_{\xi}^{2}\Phi_{a}(\xi,Mx,M^{a}t)|=at(1-a)M^a|\xi|^{a-2}\geq cM|x|.$$
Hence, we obtain  by Lemma \ref{lem:van der} 
\begin{align*}
  |\Lambda_M|\lesssim & M|Mx|^{-\frac{1}{2}}\left(M^{-\alpha}|e^{-M^{a}\epsilon2^{a}}| + \int \big|\big((1+|M\xi|^{2})^{-\frac{\alpha}{2}}\eta e^{-M^{a}\epsilon|\xi|^{a}}\big)'\big|\,\mathrm{d}\xi\right).
\end{align*}
This together with
the elementary inequality
\begin{equation}\label{exp-est}
 e^{-y}\lesssim_{\beta}y^{-\beta}, \text{~ for any ~ } y,\beta>0,
\end{equation}
yields that
\begin{align}\nonumber
  |\Lambda_M|\lesssim &M|Mx|^{-\frac{1}{2}}\left(M^{-\alpha}M^{-a{\beta}}\epsilon^{-{\beta}}+ \int M^{a}\epsilon|\xi|^{a-1}(1+|M\xi|^{2})^{-\frac{\alpha}{2}}\eta e^{-M^{a}\epsilon|\xi|^{a}}\,\mathrm{d}\xi\right.\\\nonumber
&\left.\hphantom{\lesssim |Mx|^{-\frac{1}{2}}(}+ e^{-M^{a}\epsilon(\frac{1}{2})^{a}}\int\Big(|(1+|M\xi|^{2})^{-\frac{\alpha}{2}}\eta'|+ M^{2}|\xi|(1+|M\xi|^{2})^{-\frac{\alpha}{2}-1}\eta\Big)\,\mathrm{d}\xi\right)\\\nonumber
\lesssim  & M|Mx|^{-\frac{1}{2}}M^{-\alpha-a{\beta}}\epsilon^{-{\beta}}\\\label{equ:lammesgu}
 \lesssim &M|Mx|^{-\frac{1}{2}}M^{-\alpha-a{\beta}}t^{-\gamma{\beta}},
\end{align}
where we have used the fact that $\epsilon>t^\gamma$ by definition of $t$ and $\epsilon$.

{\bf Subcase 2.1:  $\gamma\in(0,1]$ and $\alpha\in(0,1)$.}  Since $t\in(0,2)$ and $|x|\leq2^{2-a}M^{a-1}t$,
we have by \eqref{equ:lammesgu}
$$|\Lambda_M| \lesssim M|Mx|^{-\frac{1}{2}}M^{-\alpha-a\beta}t^{-\beta} \lesssim M^{\frac{1}{2}-\alpha-\beta}|x|^{-\frac{1}{2}-\beta}.$$
Choosing $\beta = \frac12-\tau >0$, we have
\begin{equation}\label{equ:gamsmall}
  \|\Lambda_M  \|_{L^1(|x|\leq2^{2-a}M^{a-1}t)} \lesssim M \int_0^{2^{2-a}M^{a-1}} M^{-\frac{1}{2}-\alpha-\beta}|x|^{-\frac{1}{2}-\beta} dx\lesssim M^{a\tau-\alpha}.
\end{equation}

{\bf Subcase 2.2: $\gamma\in(1,\infty)$ and $\alpha>\frac{a}{2}(1-\frac{1}{\gamma})$.} Note that  $|x|\leq2^{2-a}M^{a-1}t$, we get by  \eqref{equ:lammesgu}
$$|\Lambda_M|\lesssim MM^{-\frac{1}{2}-\alpha-a{\beta}-(1-a)\gamma{\beta}}|x|^{-\frac{1}{2}-\gamma{\beta}}.$$
Taking ${\beta} = \frac{1}{2\gamma}-\tau >0$, we have
$$\|\Lambda_M  \|_{L^1(|x|\leq2^{2-a}M^{a-1}t)} \lesssim M \int_0^{2^{2-a}M^{a-1}} M^{-\frac{1}{2}-\alpha-a{\beta}-(1-a)\gamma{\beta}}|x|^{-\frac{1}{2}-\gamma{\beta}}  dx\lesssim M^{\frac{a}2(1-\frac1\gamma)-\alpha+a\tau}.$$
This estimate together with \eqref{equ:gamsmall} and \eqref{equ:lambmlarge} implies that
\begin{align}
  \sum_{M\geq1}\|\Lambda_M\|_{L^1(\R)}\lesssim & \sum_{M\geq1}\big(M^{-\alpha}+M^{a\tau-\alpha}+M^{\frac{a}2(1-\frac1\gamma)-\alpha+a\tau}\big)<+\infty\label{0<a<1}
\end{align}
provided that $\alpha>\frac{a}{2}(1-\frac{1}{\gamma})_+$ and taking $\tau>0$ small enough. 

\subsubsection{\bf Estimation for $\Lambda_{M}$ with $a=1$}

In this case,
\begin{align*}
  \Lambda_M
  \triangleq&M\int e^{i\Phi(\xi,Mx,Mt)} e^{-\epsilon M|\xi|}(1+|M\xi|^{2})^{-\frac{\alpha}{2}}\eta(\xi)\mu\big(\tfrac{M\xi}{N}\big) d\xi
\end{align*}
with $\Phi(\xi,x,t)=t|\xi|-x\xi.$ Through the direct computation, we have
$$\partial_\xi\Phi(\xi,Mx,Mt)=atM|\xi|^{-1}\xi-Mx.$$

We divide the two cases to estimate $\Lambda_{M}$.

{\bf Case 1: $|x|\geq2t$.}
Note that the argument in case 1 in section 2.2.1 also holds for $a=1$, that is, the estimate \eqref{equ:lambmlarge} holds for $a=1$.

{\bf Case 2: $|x|\leq2t$.}
The elementary inequality
\begin{equation}\label{exp-est}
 e^{-y}\lesssim_{\beta}y^{-\beta}, \text{~ for any ~ } y,\beta>0,
\end{equation}
yields that
\begin{align}
|\Lambda_{M}|
&\lesssim M\int e^{-\epsilon M|\xi|}(1+|M\xi|^{2})^{-\frac{\alpha}{2}}\eta(\xi)\,\mathrm{d}\xi\nonumber\\
&\le M\int_{1}^{4}(\epsilon M|\xi|)^{-\beta}(1+|M\xi|^{2})^{-\frac{\alpha}{2}}\,\mathrm{d}\xi\nonumber\\
&\lesssim M\int_{1}^{4}t^{-\gamma\beta}M^{-\beta-\alpha}|\xi|^{-\beta-\alpha}\,\mathrm{d}\xi\nonumber\\
&\lesssim M^{1-\beta-\alpha}t^{-\gamma\beta}\label{equ:lammesgu-1}
\end{align}
where we have used the fact that $\epsilon>t^\gamma$ by definition of $t$ and $\epsilon$.

{\bf Subcase 2.1:  $\gamma\in(0,1]$ and $\alpha\in(0,1)$.}  Since $t\in(0,2)$ and $|x|\leq2t$,
we have by \eqref{equ:lammesgu-1}
$$|\Lambda_M| \lesssim M^{1-\beta-\alpha}t^{-\beta} \lesssim M^{1-\beta-\alpha}|x|^{-\beta}.$$
Choosing $\beta = 1-\tau >0$, we have
\begin{equation}\label{equ:gamsmall-1}
  \|\Lambda_M  \|_{L^1(|x|\leq2t)} \lesssim \int_0^{4} M^{1-\beta-\alpha}|x|^{-\beta}\,\mathrm{d}x\lesssim M^{\tau-\alpha}.
\end{equation}

{\bf Subcase 2.2: $\gamma\in(1,\infty)$ and $\alpha>\frac{a}{2}(1-\frac{1}{\gamma})$.} Note that  $|x|\leq2t$, we get by  \eqref{equ:lammesgu-1}
$$|\Lambda_M|\lesssim M^{1-\beta-\alpha}|x|^{-\gamma\beta}.$$
Taking ${\beta} = \frac{1}{\gamma}-\tau >0$, we have
$$\|\Lambda_M  \|_{L^1(|x|\leq2t)} \lesssim \int_0^{4} M^{1-\beta-\alpha}|x|^{-\gamma\beta}
\,\mathrm{d}x \lesssim M^{(1-\frac1\gamma)-\alpha+\tau}.$$
This estimate together with \eqref{equ:gamsmall-1} and \eqref{equ:lambmlarge} implies that
\begin{align}
  \sum_{M\geq1}\|\Lambda_M\|_{L^1(\R)}\lesssim & \sum_{M\geq1}\big(M^{-\alpha}+M^{\tau-\alpha}+M^{(1-\frac1\gamma)-\alpha+\tau}\big)<+\infty\label{a=1}
\end{align}
provided that $\alpha>(1-\frac{1}{\gamma})_+$ and taking $\tau>0$ small enough. 

Therefore, combining the estimate \eqref{0<a<1} and \eqref{a=1}, we conclude the proof of Lemma \ref{maximal estimate}(i).

\subsection{Proof of the local estimates in Lemma \ref{maximal estimate} }\label{pf-local}

\subsubsection{{\bf Local estimate for $0<a<1$.}}

  Let $t=t_1-t_2$ and $\epsilon=t_1^\gamma+t_2^\gamma$ as in the above subsection,  $F(x,t,\xi) = t|\xi|^{a} - x\xi$ and $G(\xi) = (1+\xi^{2})^{-\frac{\alpha}{2}}e^{-\epsilon|\xi|^{a}}(1-\chi(\xi))\mu(\frac{\xi}{N})$, then it is equivalent to show for $x\in B(0,1)$
 \begin{equation}\label{equ:locred}
 \Big|\int_{\mathbb{R}}e^{iF(x,t,\xi)}G(\xi)\,\mathrm{d}\xi\Big|\leq C  \begin{cases}
  |x|^{\alpha-1}\quad\text{if}\quad \gamma\in(0,1),\\
  |x|^{\alpha-1} + |x|^{-\sigma}\quad\text{if}\quad \gamma>1,
  \end{cases}
 \end{equation}
where $\sigma$ is defined as in \eqref{equ:sigadef}. To do this, we
%
split the integral into two parts as $A + B$, where
\begin{align*}
&A = \int_{|\xi|\le |x|^{-1}}e^{iF(x,t,\xi)}G(\xi)\,\mathrm{d}\xi,\\
&B = \int_{|\xi|\ge |x|^{-1}}e^{iF(x,t,\xi)}G(\xi)\,\mathrm{d}\xi.
\end{align*}
First, it is easy to see that
\begin{equation}\label{A1}
|A| \lesssim \int_{|\xi|\le |x|^{-1}}(1+\xi^{2})^{-\frac{\alpha}{2}}\,\mathrm{d}\xi \lesssim |x|^{\alpha-1}. 
\end{equation}

It remains to estimate $B$.
By symmetry, we just need to consider the positive part $\{\xi: \xi > |x|^{-1}\}$ of the integral region for $B$. We consider the following two cases.

 {\bf Case 1: $|x|^{a} \ge 2at$.} By a direct calculation, we see that $$ \partial_\xi F(x,t,\xi) = at\xi^{a-1} -x,$$
  and $\partial_\xi F(x,t,\xi)$ is monotonic with respect to $\xi$. Then the conditions $\xi \ge |x|^{-1}$ and $\frac{at}{|x|^{a-1}}\le \frac{|x|}{2}$  imply that
$$
|\partial_\xi F(x,t,\xi)| = |at\xi^{a-1}-x| \ge |x| - |at\xi^{a-1}| \ge \frac{|x|}{2}.
$$
By Lemma \ref{lem:van der}, we have
$$
\left|\int_{\xi>|x|^{-1}}e^{iF(x,t,\xi)}G(\xi)\,\mathrm{d}\xi\right| \lesssim \frac{1}{|x|}\left(\sup_{\xi>|x|^{-1}}|G(\xi)|+\int_{\xi>|x|^{-1}}|G'(\xi)|\,\mathrm{d}\xi\right).
$$
It is easy to see that
\begin{equation}\label{G-estimate}
|G(\xi)| \lesssim (1+\xi^{2})^{-\frac{\alpha}{2}} \lesssim |x|^{\alpha},\quad for\,\, \xi>|x|^{-1}.
\end{equation}
Define $h_{\epsilon}(\xi) = e^{-\epsilon|\xi|^{a}}$,  then we have
\begin{align*}
G'(\xi) = &2\xi\big(-\tfrac{\alpha}{2}\big)(1+\xi^{2})^{-\frac{\alpha}{2}-1}h_{\epsilon}(\xi)(1-\chi(\xi))\mu\big(\tfrac{\xi}{N}\big) + (1+\xi^{2})^{-\frac{\alpha}{2}}h'_{\epsilon}(\xi)(1-\chi(\xi))\mu\big(\tfrac{\xi}{N}\big) \\
&-(1+\xi^{2})^{-\frac{\alpha}{2}}h_{\epsilon}(\xi)\chi'(\xi)\mu\big(\tfrac{\xi}{N}\big)+ (1+\xi^{2})^{-\frac{\alpha}{2}}h_{\epsilon}(\xi)\tfrac{1}{N}(1-\chi(\xi))\mu'\big(\tfrac{\xi}{N}\big).
\end{align*}
Since $|h'_{\epsilon}(\xi)| \lesssim \frac{1}{\xi}$ with constant independent of $\epsilon$, then one can
obtain
\begin{align*}
|G'(\xi)|
&\le \alpha \xi (1+\xi^{2})^{-\frac{\alpha}{2}-1}h_{\epsilon}(\xi)\mu(\frac{\xi}{N})
+ (1+\xi^{2})^{-\frac{\alpha}{2}}|h'_{\epsilon}(\xi)|\mu(\frac{\xi}{N})\\
&\quad +(1+\xi^{2})^{-\frac{\alpha}{2}}h_{\epsilon}(\xi)|\chi'(\xi)|\mu(\frac{\xi}{N})
+ (1+\xi^{2})^{-\frac{\alpha}{2}}h_{\epsilon}(\xi)\frac{1}{N}|\mu'(\frac{\xi}{N})|\\
&\lesssim \xi^{-\alpha-1} + \frac{\xi^{-\alpha}}{N}|\mu'(\frac{\xi}{N})|\\
&\lesssim \xi^{-\alpha-1}.
\end{align*}
From above estimates, we have
$$
\int_{\xi>|x|^{-1}}|G'(\xi)|\,\mathrm{d}\xi \le \int_{|x|^{-1}}^{\infty}\xi^{-\alpha-1}\,\mathrm{d}\xi \le |x|^{\alpha}.
$$
This implies that
\begin{equation}\label{B}
|B| \lesssim |x|^{\alpha-1}.
\end{equation}

{\bf Case 2: $|x|^{a}\le 2at$.}
We split the integral region of $B$ into three parts as follows
\begin{align*}
&I_{1} = \{\xi\ge |x|^{-1}: \xi\le \delta \rho\};\\
&I_{2} = \{\xi\ge |x|^{-1}: \xi\in[\delta\rho,\tfrac{\rho}{\delta}]\};\\
&I_{3} = \{\xi\ge |x|^{-1}: \xi\ge \tfrac{\rho}{\delta} \},
\end{align*}
where $\delta$ is a small constant and $\rho =\big(\tfrac{|x|}{ta}\big)^{\frac{1}{a-1}}$.
Let 
$$
J_{j} = \int_{I_{j}} e^{i(t|\xi|^{a}-x\xi)} (1+\xi^{2})^{-\frac{\alpha}{2}} e^{-\epsilon|\xi|^{a}}(1-\chi(\xi))\mu(\frac{\xi}{N})\,\mathrm{d}\xi,\qquad j\in\{1,2,3\}.
$$
For $\xi\in I_{1}$, we have $at\xi^{a-1}\ge\delta^{a-1}|x|\ge2|x|$, then
$$
|\partial_{\xi}F(x,t,\xi)| = |at\xi^{a-1} - x|\ge|x|.
$$
For $\xi\in I_{3}$, there is $at\xi^{a-1} \le \delta^{1-a}|x|\le\frac{1}{2} |x|$, then
$$
|\partial_{\xi}F(x,t,\xi)| = |at\xi^{a-1} - x|\ge \frac{1}{2} |x|.
$$
It follows from the proof as in Case 1 that
$$
\sup_{\xi>|x|^{-1}}|G(\xi)| + \int_{\xi>|x|^{-1}}|G'(\xi)|\,\mathrm{d}\xi \lesssim |x|^{\alpha}.
$$
By Lemma \ref{lem:van der}, we have
\begin{equation}\label{J1-J3}
|J_{1}|, |J_{3}| \lesssim |x|^{-1}|x|^{\alpha} = |x|^{\alpha-1}.
\end{equation}

Finally we estimate the integral $J_{2}$. Since $\rho$ might be the critical point of $F$, we consider $$\partial_{\xi}^{2}F(x,t,\xi) = a(a-1)t\xi^{a-2}.$$ Using $\xi\thicksim \rho$ instead of $\xi>|x|^{-1}$, then
$$
|\partial_{\xi}^{2}F(x,t,\xi)| \ge t^{\frac{1}{a-1}}|x|^{\frac{a-2}{a-1}}.
$$
As \eqref{G-estimate}, we have
$$
\sup_{\xi\in I_{2}}|G(\xi)| \lesssim \rho^{-\alpha}e^{-\delta^{a}\epsilon\rho^{a}}.
$$
Since $|G'(\xi)| \lesssim \rho^{-\alpha}|h'_{\epsilon}(\xi)| + \rho^{-\alpha-1}h_{\epsilon}(\delta\rho)$, then we have
\begin{align*}
\int_{I_{2}}|G'(\xi)|\,\mathrm{d}\xi
&\lesssim\rho^{-\alpha}\int_{\delta\rho}^{\tfrac{\rho}{\delta}}|h'_{\epsilon}|\,\mathrm{d}\xi + \int_{\delta\rho}^{\tfrac{\rho}{\delta}}\rho^{-\alpha-1}h_{\epsilon}(\delta\rho)\,\mathrm{d}\xi\\
&=-\rho^{-\alpha}\int_{\delta\rho}^{\tfrac{\rho}{\delta}}h_{\epsilon}'(\xi)\,\mathrm{d}\xi + \int_{\delta\rho}^{\tfrac{\rho}{\delta}}\rho^{-\alpha-1}h_{\epsilon}(\delta\rho)\,\mathrm{d}\xi\\
&\thickapprox\rho^{-\alpha}e^{-\delta^{a}\epsilon\rho^{a}}.
\end{align*}
By Lemma \ref{lem:van der}, we have
\begin{align*}
|J_{2}|&\lesssim t^{-\frac{1}{2(a-1)}}|x|^{-\frac{a-2}{2(a-1)}}\left(\sup_{\xi\in I_{2}}|G(\xi)| + \int_{I_{2}}|G'(\xi)|\,\mathrm{d}\xi\right)\\
&\lesssim t^{-\frac{1}{2(a-1)}}|x|^{-\frac{a-2}{2(a-1)}}\rho^{-\alpha}e^{-\delta^{a}\epsilon\rho^{a}}\\
&\lesssim t^{\frac{1}{a-1}(\alpha-\frac{1}{2})}|x|^{\frac{1}{a-1}(-\alpha-\frac{1}{2}(a-2))}e^{-\delta^{a}(t_{1}^{\gamma}+t_{2}^{\gamma})
|x|^{\frac{a}{a-1}}t^{-\frac{a}{a-1}}}
\end{align*}
Note that $t_{1}^{\gamma}+t_{2}^{\gamma}\gtrsim c_{0}(t_{1}+t_{2})^{\gamma}\ge c_{0}t^{\gamma}$ with $c_{0} = 2^{-\gamma}$, then
$$
|J_{2}|\lesssim t^{\frac{1}{a-1}(\alpha-\frac{1}{2})}|x|^{\frac{1}{a-1}(-\alpha-\frac{1}{2}(a-2))}e^{-\delta^{a}c_{0}t^{\gamma-\frac{a}{a-1}}|x|^{\frac{a}{a-1}}}.
$$
To obtain the estimate of $J_{2}$, we consider the following two subcases.
\begin{enumerate}
  \item For $\alpha\in [\frac12,1),$ we have  $ \frac{1}{a-1}(\alpha-\frac{1}{2}) \leq 0$.
Since $|x|^{a}\le 2ta$, we have
\begin{equation}\label{J2-a>1/2}
|J_{2}|\lesssim t^{\frac{1}{a-1}(\alpha-\frac{1}{2})}|x|^{\frac{1}{a-1}(-\alpha-\frac{1}{2}(a-2))}
 \lesssim |x|^{\alpha-1}.
\end{equation}
  \item
For $\alpha \in ( \frac12a (1-\frac{1}{\gamma})_+, \frac12)$, by the inequality \eqref{exp-est},
we have
$$
|J_{2}|\lesssim t^{\frac{1}{a-1}(\alpha-\frac{1}{2})}|x|^{\frac{1}{a-1}(-\alpha-\frac{1}{2}(a-2))}t^{-\beta(\gamma-\frac{a}{a-1})}|x|^{-\frac{\beta a}{a-1}}.
$$
Choose $\beta$ such that  $\frac{1}{a-1}(\alpha-\frac{1}{2})=\beta(\gamma-\frac{a}{a-1})$, i.e., $\beta=\frac{\alpha-\frac{1}{2}}{(a-1)\gamma-a}$. In fact since $\gamma>0, \alpha<\frac{1}{2}$, we have $\beta >0$. Then
\begin{equation}\label{J2-a<1/2}
|J_{2}|\lesssim |x|^{-\sigma},
\end{equation}
where $\sigma = \frac{1}{a-1}(\alpha + \frac{1}{2}(a-2)+\frac{a(\alpha-\frac{1}{2})}{(a-1)\gamma-a})$.
When $\alpha>\frac{1}{2}a(1-\frac{1}{\gamma})$, we have $\sigma<1$.
\end{enumerate}

In summary, collect the estimates \eqref{A1}, \eqref{B}, \eqref{J1-J3}, \eqref{J2-a>1/2} and \eqref{J2-a<1/2}, then we obtain the local estimates \eqref{equ:kxcho} of Lemma \ref{maximal estimate}.

\subsubsection{{\bf Local estimate for $a=1$.}}

Choose the function $\rho(\xi) \in C^{\infty}_{c}(\mathbb{R})$ such that
\begin{equation}
\rho(\xi)=
\begin{cases}
&1,\quad \text{if}\,\, |\xi|<1;\\
&0,\quad \text{if}\,\, |\xi|>2.
\end{cases}
\end{equation}

Split the integral \eqref{goal1} into two parts as follows
\begin{align*}
&\tilde{A}= \int e^{i(t_{1}-t_{2})|\xi|}e^{-ix\xi}e^{-(t_{1}^{\gamma}+t_{2}^{\gamma})|\xi|}\rho(\tfrac{\xi}{|x|^{-\gamma}})
(1+\xi^{2})^{-\frac{\alpha}{2}}(1-\chi(\xi))\mu\big(\tfrac{\xi}{N}\big)\,\mathrm{d}\xi,\\
&\tilde{B}= \int e^{i(t_{1}-t_{2})|\xi|}e^{-ix\xi}e^{-(t_{1}^{\gamma}+t_{2}^{\gamma})|\xi|}(1-\rho(\tfrac{\xi}{|x|^{-\gamma}}))
(1+\xi^{2})^{-\frac{\alpha}{2}}(1-\chi(\xi))\mu\big(\tfrac{\xi}{N}\big)\,\mathrm{d}\xi.
\end{align*}
It is easy to see that
\begin{equation}\label{A}
\begin{split}
|\tilde{A}|
&\le \int_{|\xi|<2|x|^{-\gamma}}(1+\xi^{2})^{-\frac{\alpha}{2}}\,\mathrm{d}\xi
\lesssim |x|^{\gamma(\alpha-1)},
\end{split}
\end{equation}
for $x\in B(0,1)$.

Next we turn to look at the integral $B$. 
We consider the following two cases. 

{\bf{Case 1, $|x|<2t$.}}
In this case, we have
\begin{equation}\label{c1}
\begin{split}
|\tilde{B}|
&\le \int_{|\xi|>|x|^{-\gamma}}e^{-(t^{\gamma}_{1}+t^{\gamma}_{2})|\xi|}(1+|\xi|^{2})^{-\frac{\alpha}{2}}
(1-\chi(\xi))\mu\big(\tfrac{\xi}{N}\big)\,\mathrm{d}\xi\\
&\lesssim \int_{|\xi|>|x|^{-\gamma}}(t^{\gamma}_{1}+t^{\gamma}_{2})^{-\beta}|\xi|^{-\alpha-\beta}\,\mathrm{d}\xi\\
&\lesssim t^{-\gamma\beta} \int_{|\xi|>|x|^{-\gamma}}|\xi|^{-\alpha-\beta}\,\mathrm{d}\xi\\
&\lesssim |x|^{-\gamma\beta}|x|^{\gamma(\alpha+\beta-1)} \lesssim |x|^{\gamma(\alpha-1)}.
\end{split}
\end{equation}

{\bf{Case 2, $|x|>2t$.}}
Choose the functions $\chi_{1}(\xi)$, $\chi_{2}(\xi)$ $\in C^{\infty}_{c}(\mathbb{R})$ such that
\begin{equation}
\chi_{1}(\xi)=
\begin{cases}
&1,\quad \text{if}\,\, \xi<-1;\\
&0,\quad \text{if}\,\, \xi>0,
\end{cases}
\end{equation}
and
\begin{equation}
\chi_{2}(\xi) = \chi_{1}(-\xi).
\end{equation}

Then we have
\begin{align*}
|\tilde{B}|
&= \left|\int e^{it|\xi|}e^{ix\xi}(1-\rho(\tfrac{\xi}{|x|^{-\gamma}}))e^{-(t_{1}^{\gamma}+t_{2}^{\gamma})|\xi|}(1+\xi^{2})^{-\frac{\alpha}{2}}
(1-\chi(\xi))\mu\big(\tfrac{\xi}{N}\big)\,\mathrm{d}\xi\right|\\
&\le \left|\int e^{i(x-t)\xi}(1-\rho(\tfrac{\xi}{|x|^{-\gamma}}))\chi_{1}(\tfrac{\xi}{|x|^{-\gamma}})e^{-(t_{1}^{\gamma}+t_{2}^{\gamma})|\xi|}
(1+\xi^{2})^{-\frac{\alpha}{2}}(1-\chi(\xi))\mu\big(\tfrac{\xi}{N}\big)\,\mathrm{d}\xi\right|\\
& \phantom{\le}+ \left|\int e^{i(x+t)\xi}(1-\rho(\tfrac{\xi}{|x|^{-\gamma}}))\chi_{2}(\tfrac{\xi}{|x|^{-\gamma}})e^{-(t_{1}^{\gamma}+t_{2}^{\gamma})|\xi|}
(1+\xi^{2})^{-\frac{\alpha}{2}}(1-\chi(\xi))\mu\big(\tfrac{\xi}{N}\big)\,\mathrm{d}\xi\right|\\
&= |\tilde{B}_{1}| + |\tilde{B}_{2}|.
\end{align*}
It suffices to estimate $|\tilde{B}_{1}|$, since the estimate for $|\tilde{B}_{2}|$ is similar.

For $|\tilde{B}_{1}|$, we have
\begin{align*}
|\tilde{B}_{1}|
&\le \left|\int e^{i(x-t)\xi}e^{-(t_{1}^{\gamma}+t_{2}^{\gamma})|\xi|}\chi_{1}(\tfrac{\xi}{|x|^{-\gamma}})(1+\xi^{2})^{-\frac{\alpha}{2}}
(1-\chi(\xi))\mu\big(\tfrac{\xi}{N}\big)\,\mathrm{d}\xi\right|\\
&\phantom{\le} + \left|\int e^{i(x-t)\xi}e^{-(t_{1}^{\gamma}+t_{2}^{\gamma})|\xi|}\rho(\tfrac{\xi}{|x|^{-\gamma}})\chi_{1}(\tfrac{\xi}{|x|^{-\gamma}})
(1+\xi^{2})^{-\frac{\alpha}{2}}(1-\chi(\xi))\mu\big(\tfrac{\xi}{N}\big)\,\mathrm{d}\xi\right|\\
&= |\tilde{B}_{11}|+|\tilde{B}_{12}|.
\end{align*}
Note that if $0<\sigma<1$, we  have the estimate for the Bessel potential of order $\sigma$
\begin{equation}\label{bessel}
\left|\int e^{ix\xi}(1+|\xi|^{2})^{-\frac{\sigma}{2}}\,\mathrm{d}\xi\right| \lesssim |x|^{\sigma-1}, \quad \text{for} \,\,x\in \mathbb{R},
\end{equation}
which can be found in \cite{Glafakos M}.
Then
\begin{align*}
|\tilde{B}_{11}|
&= |\mathcal{F}^{-1}(e^{-(t^{\gamma}_{1}+t^{\gamma}_{2})|\xi|}\chi_{1}(\tfrac{\xi}{|x|^{-\gamma}})
(1+|\xi|^{2})^{-\tfrac{\alpha}{2}}(1-\chi(\xi))\mu\big(\tfrac{\xi}{N}\big))(x-t)|\\
&\lesssim \mathcal{M}(\mathcal{F}^{-1}((1+|\xi|^{2})^{-\tfrac{\alpha}{2}}))(x-t)\\
&\lesssim |x-t|^{\alpha-1}\sim  |x|^{\alpha-1},
\end{align*}
and
\begin{align*}
|\tilde{B}_{12}|
&\le \int_{|\xi|<2|x|^{-\gamma}}(1+\xi^{2})^{-\frac{\alpha}{2}}\,\mathrm{d}\xi
\lesssim |x|^{\gamma(\alpha-1)}.
\end{align*}

Thus
\begin{equation}\label{b1}
|\tilde{B}_{1}| \lesssim |x|^{\alpha-1} +|x|^{\gamma(\alpha-1)}.
\end{equation}
Similarly, we have
\begin{equation}\label{b2}
|\tilde{B}_{2}| \lesssim |x|^{\alpha-1} +|x|^{\gamma(\alpha-1)}.
\end{equation}
By \eqref{A}, \eqref{c1}, \eqref{b1} and \eqref{b2}, we obtain that
$$
{\rm LHS} \,\,\text{of}\,\, \eqref{goal1} \lesssim |x|^{\alpha-1} +|x|^{\gamma(\alpha-1)}
$$
for $x\in B(0,1)$.

The proof is completed.

\section{$L^2$-maximal estimates and  almost everywhere convergence }

In this section, we first prove Theorem \ref{main-thm-1} by employing the estimates established in previous section.
Second, we give the proof of the almost everywhere convergence result in Corollary \ref{point}  by a standard method. %
\vskip 0.2in

\subsection{Proof of Theorem $\ref{main-thm-1}$}

First, it is easy to see that third part of Theorem \ref{main-thm-1}
follows from Remark \ref{remark-Poisson} and boundedness of the Hardy-Littlewood operator $\mathcal{M}$ .

(i) Next, we turn to prove Theorem \ref{main-thm-1}(i) by the local estimate \eqref{equ:kxcho}. We first consider the case that $0<a<1.$

Let  $\eta \in \mathcal{S}(\mathbb{R})$ be a  positive, even function with $\text{supp~}\,\eta \subset [-1, 1]$ and $\eta =1$ in $[-\frac{1}{2}, \frac{1}{2}]$. Fix a measurable function $t = t(x) : \mathbb{R}\rightarrow (0,1)$, define for each $N \in \mathbb{N}$
\begin{equation}\label{tx}
P^{t(x)}_{a,\gamma,N}f(x) =\eta (\frac{x}{N})  \int_{\mathbb{R}}\hat{f}(\xi)e^{it(x)|\xi|^{a}}
e^{-t(x)^{\gamma}|\xi|^{a}}e^{ix\xi}\eta(\frac{\xi}{N})\,\mathrm{d}\xi. 
\end{equation}
Thus, we have $P^{t(x)}_{a,\gamma,N}f(x) \rightarrow P^{t(x)}_{a,\gamma}f(x)$, as $N \rightarrow\infty$.
From Fatou's lemma, it suffices to prove for $s > \frac{1}{4}a(1-\frac{1}{\gamma})_{+} $ ~and~$ N\in \mathbb{N}$,
\begin{equation}\label{KSP}
\left\|P^{t(x)}_{a,\gamma,N}f(x)\right\|_{L^{2}(B(0,1))} \le C \|f\|_{H^{s}(\mathbb{R})},
\end{equation}
with some  constant $C$ depending on $a,\gamma,s$ and independent of $f$ and $N.$
And by duality, it is equivalent to prove for $\forall g(x) \in L^{2}(B(0,1))$ with $\|g\|_{L^{2}(B(0,1))} = 1$,
\begin{equation}\label{equ:equiksp}
  \left|\int_{B(0,1)}(P^{t(x)}_{a,\gamma,N}f(x))\overline{g(x)}\,\mathrm{d}x\right| \lesssim \|f\|_{H^{s}(\mathbb{R})}.
\end{equation}
Take the function $\chi(\xi)\in S(\mathbb{R})$ such that
$$
\chi(\mathbb{R}) \subset [0,1], \,\,\chi(\xi) = 1 \,\,\text{on} \,\,[-1,1], \,\,{\rm supp}\,\, \chi(\xi)\subset[-2,2].
$$
Then by Fubini's theorem and Cauchy-Schwarz inequality, we have
\begin{align}\nonumber
   & \left|\int_{B(0,1)}\big(P^{t(x)}_{a,\gamma,N}f(x)\big)\right.\left.\vphantom{\int_{B(0,1)}}\overline{g(x)}\,\mathrm{d}x\right|^{2} \\\nonumber
=  & \left|\int_{\mathbb{R}}\hat{f}(\xi)(1+\xi^{2})^{\frac{s}{2}}(1+\xi^{2})^{-\frac{s}{2}}\eta\big(\tfrac{\xi}{N}\big)
     \int_{B(0,1)}e^{it(x)|\xi|^{a}}e^{-t(x)^{\gamma}|\xi|^{a}}e^{ix\xi}g(x)\eta\big(\tfrac{x}{N}\big)
     \,\mathrm{d}x\mathrm{d}\xi \right|^{2}\\\nonumber
\le& \|f\|^{2}_{H^{s}(\mathbb{R})}\left|\int_{\mathbb{R}}(1+\xi^{2})^{-s}\eta\big(\tfrac{\xi}{N}\big)
     \int_{B(0,1)}\int_{B(0,1)}e^{i(t(x)-t(y))|\xi|^{a}}
    e^{-(t(x)^{\gamma}+t(y)^{\gamma})|\xi|^{a}}e^{i(x-y)\xi}\right. \\\nonumber
   &  \left.\vphantom{\int_{B(0,1)}}\phantom{\|f\|_{H^{s}(\mathbb{R})}
      \left|\int_{\mathbb{R}}(1+\xi^{2})^{-s}\eta\big(\tfrac{\xi}{N}\big)\int_{B(0,1)}\int_{B(0,1)}\right.}
      \times g(x)\overline{g(y)}\eta\big(\tfrac{x}{N}\big)\eta\big(\tfrac{y}{N}\big)\,
      \mathrm{d}x\mathrm{d}y\mathrm{d}\xi\right|\\\nonumber
\le& \|f\|^{2}_{H^{s}(\mathbb{R})} \int_{B(0,1)}\int_{B(0,1)}\big|g(x)\overline{g(y)}\big|\cdot
     \left|\int_{\mathbb{R}}e^{i[(t(x)-t(y))|\xi|^{a}-(y-x)\xi]}(1+\xi^{2})^{-s}\right. \left.e^{-(t(x)^{\gamma}+t(y)^{\gamma})|\xi|^{a}}\eta\big(\tfrac{\xi}{N}\big)^{2}\,\mathrm{d}\xi\right|
     \,\mathrm{d}x\mathrm{d}y\\\nonumber
\le& \|f\|^{2}_{H^{s}(\mathbb{R})} \int_{B(0,1)}\int_{B(0,1)}\big|g(x)\overline{g(y)}\big|\cdot
     \left(\left|\int_{\mathbb{R}}e^{i[(t(x)-t(y))|\xi|^{a}-(y-x)\xi]}(1+\xi^{2})^{-s}\chi(\xi) e^{-(t(x)^{\gamma}+t(y)^{\gamma})|\xi|^{a}}\eta\big(\tfrac{\xi}{N}\big)^{2}\,\mathrm{d}\xi\right|\right.\\\nonumber
&\phantom{\|f\|^{2}_{H^{s}(\mathbb{R})} \int_{B(0,1)}}+ \left.\sup_{t_{1},t_{2}\in(0,1)}\left|\int_{\mathbb{R}}e^{i[(t_1-t_2)|\xi|^{a}-(y-x)\xi]}(1+\xi^{2})^{-s}(1-\chi(\xi))
      e^{-(t_1^{\gamma}+t_2^{\gamma})|\xi|^{a}}\eta\big(\tfrac{\xi}{N}\big)^{2}\,\mathrm{d}\xi\right|\right)
     \,\mathrm{d}x\mathrm{d}y\\\label{equ:kxyest}
\lesssim& \|f\|^{2}_{H^{s}(\mathbb{R})}\int_{B(0,1)}
     \int_{B(0,1)}\big|g(x)\overline{g(y)}\big|\cdot(1+|K(x-y)|)
     \,\mathrm{d}x\mathrm{d}y.
\end{align}
On the other hand, applying
 Lemma \ref{maximal estimate} with $\alpha = 2s$, $\mu = \eta^{2}$, $t_{1} = t(x)$, $t_{2} = t(y)$, 
we get
\begin{align*}
\int_{B(0,1)}\int_{B(0,1)}\big|g(x)\overline{g(y)}\big|\cdot|K(x-y)|\,\mathrm{d}x\mathrm{d}y
\leq& \int_{B(0,1)}\int_{B(0,1)}\big|g(x)\overline{g(y)}\big|\big( |x-y|^{\alpha-1} + |x-y|^{-\sigma}\big)\,\mathrm{d}x\mathrm{d}y\\
=&
\int_{B(0,1)}\int_{B(0,1)}\big|g(x)\overline{g(y)}\big|\cdot|x-y|^{\alpha-1}\,\mathrm{d}x\mathrm{d}y\\
&+ \int_{B(0,1)}\int_{B(0,1)}\big|g(x)\overline{g(y)}\big|\cdot|x-y|^{-\sigma}\,\mathrm{d}x\mathrm{d}y\\
=& M_{1} + M_{2}.
\end{align*}
Since $\frac{1}{2}a(1-\frac{1}{\gamma})_{+}<\alpha = 2s <\frac12$, by H\"{o}lder inequality and Hardy-Littlewood-Sobolev inequality, we estimate
\begin{align*}
M_{1}
\le& \|g(x)\|_{L^{2}(B(0,1))}\left\|\int_{B(0,1)} \frac{|g(y)|}{|x-y|^{1-\alpha}}\,\mathrm{d}y\right\|_{L^{2}(B(0,1))}\\
\le& \|g(x)\|_{L^{2}(B(0,1))}\|g(y)\|_{L^{\frac{2}{2\alpha+1}}(B(0,1))}\\
\le& \|g(x)\|_{L^{2}(B(0,1))}\|g(y)\|_{L^{2}(B(0,1))} \le 1.
\end{align*}
Similarly, we obtain $M_{2} \le 1$. Plugging this into \eqref{equ:kxyest}, we obtain \eqref{equ:equiksp}.

By the similar argument above, we can also obtain the proof of the case $a=1$. So we conclude the proof of
 Theorem \ref{main-thm-1}(i).

(ii) Finally, we shall prove Theorem \ref{main-thm-1}(ii).

For the case $a\in(0,1]$, split the function $f(x)$ into two functions as follows
$$
f(x) = \mathcal{F}^{-1}(\hat{f}\chi(\xi)) + \mathcal{F}^{-1}(\hat{f}(1-\chi(\xi)))\triangleq f_{1}(x) + f_{2}(x).
$$

{\bf Step 1: Estimation for $f_{1}(x)$.}
Choose the function $\psi(t)\in S(\mathbb{R})$ such that
\begin{equation}
\psi(t)= \begin{cases}
&1, \quad\text{if}\,\,\,|t|\le1,\\
&0, \quad\text{if}\,\,\,|t|\ge2.
\end{cases}
\end{equation}
Set
$$
T_{t}f(x) = \psi(t)\int e^{ix\xi}e^{it|\xi|^{a}}e^{-|t|^{\gamma}|\xi|^{a}}\hat{f}(\xi)\,\mathrm{d}\xi,
$$
where $a\in(0,1]$ and $\gamma>1$. By Plancherel's theorem, it is easy to see that
\begin{equation}\label{low-norm}
\begin{split}
 \|T_{t}f_{1}\|_{L_{t,x}^{2}(\mathbb{R}\times\mathbb{R})}
= \|\psi(t)e^{-t^{\gamma}|\xi|^{a}}\hat{f}\chi(\xi)\|_{L_{t,\xi}^{2}(\mathbb{R}\times\mathbb{R})}
\le \|\psi(t)\|_{L^{2}_{t}(\mathbb{R})}\|\hat{f}\chi(\xi)\|_{L^{2}_{\xi}(\mathbb{R})}
\lesssim \|f(x)\|_{L^{2}(\mathbb{R})}.
\end{split}
\end{equation}
We have
\begin{align*}
\partial_{t}T_{t}f(x)
&= \psi'(t)\int e^{ix\xi}e^{it|\xi|^{a}}e^{-t^{\gamma}|\xi|^{a}}\hat{f}(\xi)\,\mathrm{d}\xi\\
&\phantom{=}+\psi(t)\int e^{ix\xi}e^{it|\xi|^{a}}e^{-t^{\gamma}|\xi|^{a}}(i-\gamma |t|^{\gamma-2}t)|\xi|^{a}\hat{f}(\xi)\,\mathrm{d}\xi.
\end{align*}
Since $t\in {\rm supp}\,\,\psi \subset[-2,2]$ and $\gamma>1$, then $\big|i-\gamma |t|^{\gamma-2}t\big| \le C_{\gamma}$.
By \eqref{low-norm}, we have
\begin{equation}\label{high-norm}
\begin{split}
\|T_{t}f_{1}\|_{L^{2}_{x}(\mathbb{R},H^{1}_{t}(\mathbb{R}))}
\le{} &\|T_{t}f_{1}\|_{L_{t,x}^{2}(\mathbb{R}\times\mathbb{R})}
+\|T_{t}(\mathcal{F}^{-1}(|\xi|^{a}\hat{f}_{1}))\|_{L_{t,x}^{2}(\mathbb{R}\times\mathbb{R})}\\
\lesssim{} &\|f(x)\|_{L^{2}(\mathbb{R})} + \||\xi|^{a}\hat{f}\chi(\xi)\|_{L^{2}_{\xi}(\mathbb{R})}\\
\lesssim{} &\|f(x)\|_{L^{2}(\mathbb{R})}
\end{split}
\end{equation}
Interpolation between \eqref{low-norm} and \eqref{high-norm} yields
\begin{equation}\label{interpolation}
\|T_{t}f_{1}\|_{L^{2}_{x}(\mathbb{R},H^{r}_{t}(\mathbb{R}))}
\lesssim \|f\|_{L^{2}(\mathbb{R})},     \quad \text{for}\,\, r\in[0,1].
\end{equation}
Taking $r=\frac23$ in \eqref{interpolation}, by Sobolev inequality, we have
\begin{equation}\label{global-f1}
\begin{split}
\|\sup_{0<t<1}|P^{t}_{a,\gamma}f_{1}|\|_{L^{2}(\mathbb{R})}
\le \|T_{t}f_{1}\|_{L^{2}_{x}(\mathbb{R},H^{\frac23}_{t}(\mathbb{R}))}
\lesssim \|f\|_{L^{2}(\mathbb{R})},
\end{split}
\end{equation}
where $a\in(0,1]$ and $\gamma>1$.

{\bf Step 2: Estimation for $f_{2}$.}
Similar to (i), we only need to prove
\begin{equation}\label{global-f2}
\left\|P^{t(x)}_{a,\gamma,N}f_{2}(x)\right\|_{L^{2}(\R)} \le C \|f\|_{H^{s}(\mathbb{R})}.
\end{equation}
By duality, it is equivalent to prove for $\forall g(x) \in L^{2}(\mathbb{R})$ with $\|g\|_{L^{2}(\mathbb{R})} = 1$
\begin{equation}\label{equ:equiksp-global}
  \left|\int_{B(0,1)}\big(P^{t(x)}_{a,\gamma,N}f_{2}(x)\big)\overline{g(x)}\,\mathrm{d}x\right| \lesssim \|f\|_{H^{s}(\mathbb{R})}.
\end{equation}
Through the same argument in \eqref{equ:kxyest}, we can obtain
\begin{align}\nonumber
   & \left|\int_{\mathbb{R}}\big(P^{t(x)}_{a,\gamma,N}f_{2}(x)\big)\right.\left.\vphantom{\int_{B(0,1)}}\overline{g(x)}\,\mathrm{d}x\right|^{2} \\\nonumber
=  & \left|\int_{\mathbb{R}}\hat{f}(\xi)(1-\chi(\xi))(1+\xi^{2})^{\frac{s}{2}}(1+\xi^{2})^{-\frac{s}{2}}\eta\big(\tfrac{\xi}{N}\big)
     \int_{\mathbb{R}}e^{it(x)|\xi|^{a}}e^{-t(x)^{\gamma}|\xi|^{a}}e^{ix\xi}g(x)\eta\big(\tfrac{x}{N}\big)
     \,\mathrm{d}x\mathrm{d}\xi \right|^{2}\\\nonumber
\le& \|f\|^{2}_{H^{s}(\mathbb{R})} \int_{\mathbb{R}}\int_{\mathbb{R}}\big|g(x)\overline{g(y)}\big|\cdot
     \left|\int_{\mathbb{R}}e^{i[(t(x)-t(y))|\xi|^{a}-(y-x)\xi]}(1+\xi^{2})^{-s}\right. \left.e^{-(t(x)^{\gamma}+t(y)^{\gamma})|\xi|^{a}}(1-\chi(x))^{2}\eta\big(\tfrac{\xi}{N}\big)^{2}\,\mathrm{d}\xi\right|
     \,\mathrm{d}x\mathrm{d}y\\\nonumber
\le& \|f\|^{2}_{H^{s}(\mathbb{R})} \int_{\mathbb{R}}\int_{\mathbb{R}}\big|g(x)\overline{g(y)}\big|\cdot
     \sup_{t_{1},t_{2}\in(0,1)}\left|\int_{\mathbb{R}}e^{i[(t_1-t_2)|\xi|^{a}-(y-x)\xi]}(1+\xi^{2})^{-s}
      e^{-(t_1^{\gamma}+t_2^{\gamma})|\xi|^{a}}(1-\tilde{\chi}(\xi))\eta\big(\tfrac{\xi}{N}\big)^{2}\,\mathrm{d}\xi\right|
     \,\mathrm{d}x\mathrm{d}y\\\label{equ:kxyest-global}
\lesssim& \|f\|^{2}_{H^{s}(\mathbb{R})}\int_{\mathbb{R}}
     \int_{\mathbb{R}}\big|g(x)\overline{g(y)}\big|\cdot|K(x-y)|
     \,\mathrm{d}x\mathrm{d}y,
\end{align}
where $\tilde{\chi}(\xi) = 2\chi(\xi) - \chi^{2}(\xi) \in S(\mathbb{R})$ satisfying that
$$
\tilde{\chi}(\mathbb{R}) \subset [0,1], \,\,\tilde{\chi}(\xi) = 1 \,\,\text{on} \,\,[-1,1], \,\,{\rm supp}\,\, \tilde{\chi}\subset[-2,2].
$$
For \eqref{equ:kxyest-global}, by Lemma \ref{maximal estimate} with $\alpha = 2s$, $\mu = \eta^{2}$, $t_{1} = t(x)$, $t_{2} = t(y)$, and
Young's convolution inequality, we obtain that
$$
\left|\int_{\mathbb{R}}\big(P^{t(x)}_{a,\gamma,N}f_{2}(x)\big)\overline{g(x)}\,\mathrm{d}x\right|
\lesssim \|f\|^{2}_{H^{s}(\mathbb{R})}\|g(x)\|_{L^{2}(\mathbb{R})}\|g(y)\|_{L^{2}(\mathbb{R})}
\le \|f\|^{2}_{H^{s}(\mathbb{R})},
$$
which proves the estimate \eqref{global-f2}.

\eqref{global-f1} and \eqref{global-f2} together finish the proof of  Theorem \ref{main-thm-1}.

\subsection{The proof of Corollary \ref{point}}
First we claim that for $\forall f \in S(\mathbb{R})$, we have that
\begin{equation}\label{schwartz}
P^{t}_{a,\gamma}f(x) \rightarrow f(x), \,\,\forall x\in \mathbb{R}.
\end{equation}
In fact, through a direct computation, we have
\begin{align*}
|P^{t}_{a,\gamma}f(x) - f(x)| = \left|\int e^{ix\xi}\hat{f}(\xi)(e^{it|\xi|^{a}}e^{-t^{\gamma}|\xi|^{a}}-1)\,\mathrm{d}\xi\right|
\lesssim \int \big|\hat{f}(\xi)(e^{it|\xi|^{a}}e^{-t^{\gamma}|\xi|^{a}}-1)\big|\,\mathrm{d}\xi.
\end{align*}
We see that when $t \rightarrow 0$, there holds pointwise convergence
$$
\hat{f}(\xi)(e^{it|\xi|^{a}}e^{-t^{\gamma}|\xi|^{a}}-1) \rightarrow 0,  \,\,\,\forall \xi\in \mathbb{R}.
$$
Meanwhile, we have
$$
|\hat{f}(\xi)(e^{it|\xi|^{a}}e^{-t^{\gamma}|\xi|^{a}}-1)| \le 2|\hat{f}(\xi)| \in L^{1}.
$$
Then \eqref{schwartz} follows from the Lebesgue Dominated Convergence Theorem.

Next we consider the function $f\in H^{s}$. Since Schwartz functions are dense in $H^{s}$, then for $\forall \epsilon >0$, we have $f = g + h$, where g is the Schwartz function and $\|h\|_{H^{s}} < \epsilon$. With this decomposition of $f$, we have that
\begin{align*}
\limsup_{t \rightarrow 0}|P^{t}_{a,\gamma}f(x) -f(x)|
\le &~\limsup_{t \rightarrow 0}|P^{t}_{a,\gamma}g(x) - g(x)| + \limsup_{t \rightarrow 0}|P^{t}_{a,\gamma}h(x) - h(x)|\\
\le& \sup_{0 < t \le 1}|P^{t}_{a,\gamma}h(x)| + |h(x)|.
\end{align*}

For $\forall \lambda >0$, set
$$
E_{\lambda} = \left\{x\in \R: \limsup_{t \rightarrow 0}|P^{t}_{a,\gamma}f(x) - f(x)| > \lambda\right\}.
$$
and
$$
E = \cup_{k=1}^{\infty} E_{\frac{1}{k}}.
$$
In order to prove the convergence a.e., we just need to prove $|E| = 0$.
It is obvious that
\begin{equation}\label{decomposition}
|E_{\lambda}|
\le \left|\left\{x\in \R: \sup_{0 < t \le 1}|P^{t}_{a,\gamma}h(x)| > \frac{\lambda}{2}\right\}\right|
+ \left|\left\{x\in \R: |h(x)| > \frac{\lambda}{2}\right\}\right|.
\end{equation}

{\bf Case 1: $\gamma >1$.}
By the maximal estimate 
in Theorem \ref{main-thm-1}(ii), for $0<a<1$ with $s>\frac{1}{4}a(1-\frac1\gamma)$ and $a=1$ with $s>\frac{1}{2}(1-\frac1\gamma)$, we have
\begin{equation}\label{1-term}
\begin{split}
\left|\left\{x\in \R: \sup_{0 < t \le 1}|P^{t}_{a,\gamma}h(x)| > \tfrac{\lambda}{2}\right\}\right|
\le  \frac{\|\sup_{0 < t \le 1}|P^{t}_{a,\gamma}h(x)|\|_{L^{2}(\R)}}{(\lambda/2)^{2}} \lesssim \frac{\|h\|^{2}_{H^{s}}}{\lambda^{2}} \lesssim \frac{\epsilon^{2}}{\lambda^{2}},
\end{split}
\end{equation}
and
\begin{equation}\label{2-term}
|\{x\in \R: |h(x)| > \tfrac{\lambda}{2}\}| \le \frac{\|h\|_{L^2}^{2}}{(\lambda/2)^{2}} \le \frac{\|h\|_{H^{s}}^{2}}{(\lambda/2)^{2}} \lesssim \frac{\epsilon^{2}}{\lambda^{2}}.
\end{equation}
Therefore, for $\forall \lambda > 0$, there holds
\begin{equation}\label{two-term}
|E_{\lambda}| \lesssim \frac{\epsilon^{2}}{\lambda^{2}}
\end{equation}
for $\forall \epsilon > 0$. Then we have $|E_{\lambda}| = 0$ for any $\lambda > 0$. These imply that
$$
\left|\left\{x\in \R: \limsup_{t \rightarrow 0}|P^{t}_{a,\gamma}f(x) - f(x)| \neq 0\right\}\right| = |E| = 0.
$$

Thus for $a\in(0,1)$, we have
$$
 \lim_{t\rightarrow 0^+ } P^{t}_{a,\gamma}f(x) =f(x),\quad \textrm{for}\,\,a.e.\,x\in \R,
$$
where $f\in H^{s}$, $s>\frac{1}{4}a(1-\frac1\gamma)$ and $\gamma>1$.

For $a=1$, we have
$$
 \lim_{t\rightarrow 0^+ } P^{t}_{1,\gamma}f(x) =f(x),\quad \textrm{for}\,\,a.e.\,x\in \R,
$$
where $f\in H^{s}$, $s>\frac{1}{2}(1-\frac1\gamma)$ and $\gamma>1$.

{\bf Case 2: $0<\gamma\le1$.}
Let $f\in L^{p}(\R)$, $1\le p<\infty$. For $\forall \epsilon >0$, choose $g\in S(\R)$ such that $f=g+h$ with $\|h\|_{L^{p}} < \epsilon$.
By the estimate \eqref{maximal-est-L2} and \eqref{p=1} in Theorem \ref{main-thm-1}, we have the weak type inequality
\begin{equation}\label{weak(p,p)}
\left|\{x\in\R: |P^{\ast}_{a,\gamma}f(x)| > \lambda\}\right| < C \frac{\|f\|^{p}_{L^{p}}}{\lambda^{p}},
\end{equation}
for $f\in L^{p}(\R)$, $1\le p<\infty$, $\lambda >0$.

By the estimate \eqref{decomposition},  \eqref{2-term} and \eqref{weak(p,p)}, we have
\begin{align*}
|E_{\lambda}|
&\le \left|\left\{x\in \R: \sup_{0 < t \le 1}|P^{t}_{a,\gamma}h(x)| > \tfrac{\lambda}{2}\right\}\right|\phantom{\le}+|\{x\in \R: |h(x)| > \tfrac{\lambda}{2}\}| \\
&\lesssim  \frac{\|h\|_{L^p}^{p}}{(\lambda/2)^{p}} \lesssim \frac{\epsilon^{p}}{\lambda^{p}},
\end{align*}
then we have $|E_{\lambda}| = 0$, for any $ \lambda>0$. This yields that $|E| =0$, which means for $0<a\le1$,
\begin{equation*}
 \lim_{t\rightarrow 0^+ } P^{t}_{a,\gamma}f(x) =f(x),\quad \textrm{for}\,\,a.e.\,x\in \R,
\end{equation*}
where $f\in L^{p}$ with $1\le p<\infty$.

The proof is completed.

\section{Sharpness of the sobolev index $s$ when $0<a\leq1$}

In this section, we will show the necessary condition for the pointwise convergence, that is Theorem \ref{neg-thm}. To do this, we first
derive that the almost everywhere convergence result implies the weak boundedness of the operator $P^{*}_{a,\gamma}.$

\begin{proposition}\label{prop:poconwe}
If  the almost everywhere convergence for $P^{t}_{a,\gamma}f(x)$ with $f\in H^{s}$ holds, then,
for any $ \epsilon >0$, there exists a set $E_{\epsilon} \subset [0,1]$ with $mE_{\epsilon}> m[0,1]-\epsilon$ such that
\begin{equation}\label{equ:weak2-2}
\big|\{x\in E_{\epsilon}: |P^{*}_{a,\gamma}f(x)| > \lambda\}\big|
\le C_{\epsilon}\lambda^{-2}\|f\|^{2}_{H^{s}},
\end{equation}
for $\forall \lambda >0$ and $f\in H^{s}(\mathbb{R})$.
\end{proposition}

To prove Proposition \ref{prop:poconwe}, we need two results of Nikishin \cite{Ninkishin}.

\begin{lemma}[Nikishin \cite{Ninkishin}]\label{Nikishin}
Assume $X$ is a space with $\sigma$-finite measure, $D$ is an N-dimensional region with $mD <\infty$ and $S(D)$ denotes the set of the measurable functions on the region $D$. Assume also that $L^{p}(X)$ is separable. Let $G$ be a bounded\footnote{Here we say $G : L^{p}(X) \rightarrow S(D)$ is bounded, if for $\forall \epsilon >0$, there exists a constant $R>0$ such that for $\forall f\in L^{p}(X)$ with $\|f\|_{L^{p}} \le 1$, we have $m\{x: |Gf| \ge R\} \le \epsilon$.} hyperlinear operator from $L^{p}(X)$ into $S(D)$, which means that
$$
G(f+g) \le Gf + Gg,\quad f,\, g\in  L^{p}(X).
$$
Then for an arbitrary $\epsilon >0$ there exists a set $E_{\epsilon}\subset D$ with $mE_{\epsilon} \geq mD-\epsilon$ such that

\begin{equation}
m\{x\in E_{\epsilon}, |Gf|\geq \lambda\} \leq C_{\epsilon}\Bigl(\frac{\|f\|_{L^{p}}}{\lambda}\Bigr)^{q},
\end{equation}
for all $\lambda >0$ and $f\in L^{p}(X)$. Here $q=min(p,2)$.

\end{lemma}


The second result explains the relationship between the almost everywhere pointwise convergence and the boundedness for the relevant operator.

\begin{lemma}[Nikishin\cite{Ninkishin}]\label{Nikishin 2}
Let $T_{n} : L^{p}(X) \rightarrow S[0,1]$ be a sequence of linear operators which are continuous in measure\footnote{Each $T_{n}$ is continuous in measure, if convergence of $f_{k} \rightarrow f_{0}$ in $L^{p}(X)$ implies convergence of $T_{n}f_{k} \rightarrow T_{n}f_{0}$ in measure on $[0,1]$.}. If for each $f \in L^{p}(X)$ the $\lim_{n\rightarrow\infty}T_{n}f$ exists almost everywhere on $[0,1]$, then the operator $G$ defined by $Gf = \sup_{n}|T_{n}f|$ is hyperlinear and bounded.
\end{lemma}


Now, we apply these two lemmas to show Proposition \ref{prop:poconwe}.
Recall $$
P^{t}_{a,\gamma}f(x)=\int \hat{f}(\xi)e^{2\pi i(x\xi-t|\xi|^{\gamma})}e^{-t^{\gamma}|\xi|^{a}}\,\mathrm{d}\xi,\quad x\in\mathbb{R},~ \hat{f}\in L^{2}((1+|\xi|^{2})^{s}\mathrm{d}\xi).
$$
One can regard $P^{t}_{a,\gamma}$ as an operator on $L^{2}((1+|\xi|^{2})^s\mathrm{d}\xi)$.
By Chebyshev's inequality, one can find that $P^{t}_{a,\gamma}$ is continuous in measure. And if the almost everywhere convergence for $P^{t}_{a,\gamma}f(x)$ with $f\in H^{s}$ holds,
we have that $P^{*}_{a,\gamma}$ is bounded and hyperlinear by Lemma \ref{Nikishin 2}.
 Therefore,  Proposition \ref{prop:poconwe} follows from  Lemma \ref{Nikishin}.

\begin{proof}[{\bf The proof of Theorem \ref{neg-thm}:}]
First we consider $0<a<1$.

{\bf Case 1: $0<a<1.$}
By contradiction, we assume that
 the almost everywhere convergence \eqref{aeconv} holds for $s<\frac{1}{4}a\big(1-\frac{1}{\gamma}\big)$ with $\gamma>1$.
Then, using Proposition \ref{prop:poconwe}, we obtain for any $ \epsilon >0$, there exists a set $E_{\epsilon} \subset [0,1]$ with $mE_{\epsilon}> 1-\epsilon$ such that
\begin{equation}\label{equ:weak2-21}
\big|\{x\in E_{\epsilon}: |P^{*}_{a,\gamma}f(x)| > \lambda\}\big|
\le C_{\epsilon}\lambda^{-2}\|f\|^{2}_{H^{s}},
\end{equation}
for $\forall \lambda >0$, $f\in H^{s}(\mathbb{R})$ and $s<\frac{1}{4}a\big(1-\frac{1}{\gamma}\big)$. In the following, we will construct an counterexample to get an contradiction.

For $\nu \in(0,1)$,
choose $g_{\nu} \in S(\R)$ such that
\begin{equation*}
g_{\nu}(\xi) = \begin{cases}
1, \text{~\quad if \quad} |\xi|< \frac{1}{2}\nu^{(a-1)-\frac{a}{\gamma}};\\
0,\text{~\quad if \quad}|\xi|>\nu^{(a-1)-\frac{a}{\gamma}}.
\end{cases}
\end{equation*}
Let $\hat{f}_{\nu}(\xi) = \nu g_{\nu}(\nu\xi+\frac{1}{\nu})$, then we have
$$
\|f_{\nu}\|^{2}_{H^{s}}
= \int_{\mathbb{R}}(1+\xi^{2})^{s}|\hat{f}_{v}|^{2}\,\mathrm{d}\xi
= \nu^{2}\int_{\mathbb{R}}(1+\xi^{2})^{s}|g_{v}(\nu\xi+\frac{1}{\nu})|^{2}\,\mathrm{d}\xi
\lesssim \nu^{a-4s-\frac{a}{\gamma}}.
$$
Since $s<\frac{1}{4}a(1-\frac{1}{\gamma})$, we have
 $$\|f_{\nu}\|_{H^{s}}\rightarrow 0$$ as $\nu\rightarrow 0$.
Let $\eta = \nu\xi+\frac{1}{\nu}$, and recall that
$$
P^{t}_{a,\gamma}f_{\nu}(x) = \int_{\mathbb{R}}e^{i(x\xi+t|\xi|^{a})}e^{-t^{\gamma}|\xi|^{a}}\nu g_{\nu}(\nu\xi+\frac{1}{\nu})\,\mathrm{d}\xi,
$$
then
\begin{align*}
|P^{t}_{a,\gamma}f_{\nu}(x)|
&= \left|\int_{\mathbb{R}}e^{i(x\frac{1}{\nu}(\eta-\frac{1}{\nu})+t|\frac{1}{\nu}(\eta-\frac{1}{\nu})|^{a})}
e^{-t^{\gamma}|\frac{\eta}{\nu}-\frac{1}{\nu^{2}})|^{a}}g_{\nu}(\eta)\,\mathrm{d}\eta\right|\\
&= \left|\int_{\mathbb{R}}e^{i[x\frac{\eta}{\nu}+t|\frac{\eta}{\nu}-\frac{1}{\nu^{2}}|^{a}]}
e^{-t^{\gamma}|\frac{\eta}{\nu}-\frac{1}{\nu^{2}}|^{a}}g_{\nu}(\eta)\,\mathrm{d}\eta\right|.
\end{align*}
Let
\begin{align*}
F_{x,t,\nu}(\eta) = x\frac{\eta}{\nu}+t\left|\frac{\eta}{\nu}-\frac{1}{\nu^{2}}\right|^{a}-\frac{t}{\nu^{2a}},~
G_{t,\nu}(\eta) = t^{\gamma}\left|\frac{\eta}{\nu}-\frac{1}{\nu^{2}}\right|^{a}.
\end{align*}
By the property of the support of $g_{\nu}$, we have
$$
|P^{t}_{a,\gamma}f_{\nu}(x)| \ge \left|\int_{-\nu^{(a-1)-\frac{a}{\gamma}}}^{\nu^{(a-1)-\frac{a}{\gamma}}}
\cos(F_{x,t,\nu}(\nu))e^{-G_{t,\nu}(\eta)}g_{\nu}(\eta)\,\mathrm{d}\eta\right|.
$$
Using Taylor's formula for $|\eta|\le \nu^{(a-1)-\frac{a}{\gamma}}$, we have
$$
\left|\frac{\eta}{\nu}-\frac{1}{\nu^{2}}\right|^{a} = \frac{1}{\nu^{2a}} - \frac{a\eta}{\nu^{2a-1}}
+ \frac{a(a-1)}{2}\frac{\eta^{2}}{\nu^{2(a-1)}} + o\left(\frac{t\eta^{2}}{\nu^{2(a-1)}}\right),
$$
then
$$
F_{x,t,\nu}(\eta) = x\frac{\eta}{\nu} - \frac{ta\eta}{\nu^{2a-1}} +\frac{a(a-1)}{2}\frac{t\eta^{2}}{\nu^{2(a-1)}} + o\left(\frac{t\eta^{2}}{\nu^{2(a-1)}}\right).
$$
For $x\in[0,a\nu^{\frac{2a}{\gamma}-2(a-1)}]\subset[0,1]$, fix $t = \frac{x\nu^{2(a-1)}}{a} \in (0,1)$. Then
$$
F_{x,t,\nu}(\eta) = \frac{a-1}{2}x\eta^{2} + o(x\eta^{2}),
$$
and
$$
|F_{x,t,\nu}(\eta)| \lesssim |x\eta^{2} + o(x\eta^{2})| \lesssim a\nu^{\frac{2a}{\gamma}-2(a-1)}\nu^{2[(a-1)-\frac{a}{\gamma}]}\lesssim a <1.
$$
Similarly,
$$
G_{t,\gamma}(\eta) = t^{\gamma}\left|\frac{\eta}{\nu} - \frac{1}{\nu^{2}}\right|^{a}
= \frac{t^{\gamma}}{\nu^{2a}} - \frac{at^{\gamma}\eta}{\nu^{2a-1}} + \frac{a(a-1)}{2}\frac{t^{\gamma}\eta^{2}}{\nu^{2(a-1)}} + o\left(\frac{t^{\gamma}\eta^{2}}{\nu^{2(a-1)}}\right).
$$
Since $|\eta|\le \nu^{(a-1)-\frac{a}{\gamma}} \le \nu^{-1}$, $x\in[0,a\nu^{\frac{2a}{\gamma}-2(a-1)}]\subset[0,1]$ and $t = \frac{x\nu^{2(a-1)}}{a}$, then
\begin{align*}
|G_{t,\nu}(\eta)|
&\lesssim \frac{t^{\gamma}}{\nu^{2a}} + \frac{t^{\gamma}}{\nu^{2a}} \frac{\eta}{\nu^{-1}} + \frac{t^{\gamma}}{\nu^{2a}}\frac{\eta^{2}}{\nu^{-2}} + o(\frac{t^{\gamma}}{\nu^{2a}}\frac{\eta^{2}}{\nu^{-2}})\\
&\lesssim \frac{t^{\gamma}}{\nu^{2a}} + o(\frac{t^{\gamma}}{\nu^{2a}})\\
&\lesssim x^{\gamma}\nu^{2a\gamma-2\gamma-2a} + o(x^{\gamma}\nu^{2a\gamma-2\gamma-2a})\\
&\lesssim 1.
\end{align*}

In conclusion, when $x\in[0,a\nu^{\frac{2a}{\gamma}-2(a-1)}]$, $t = \frac{x\nu^{2(a-1)}}{a}$, and $\eta\in[-\nu^{(a-1)-\frac{a}{\gamma}},\nu^{(a-1)-\frac{a}{\gamma}}]$, there holds
\begin{align*}
\cos(F_{x,t,\nu}(\eta)) &\gtrsim C,\quad e^{-G_{t,\nu}(\eta)}\gtrsim C.
\end{align*}
Hence,
$$
|P^{*}_{a,\gamma}f_{\nu}(x)| \gtrsim \nu^{(a-1)-\frac{a}{\gamma}},\qquad for\,\,\, x\in[0,a\nu^{\frac{2a}{\gamma}-2(a-1)}].
$$
Put this inequality into  \eqref{equ:weak2-21}, we have
\begin{equation}\label{weak(2,2)}
\begin{split}
&|\{x\in E_{\epsilon}\cap[0,a\nu^{\frac{2a}{\gamma}-2(a-1)}]: |P^{*}_{a,\gamma}f_{\nu}(x)|\gtrsim \nu^{(a-1)-\frac{a}{\gamma}}\}|\\
\le& |\{x\in E_{\epsilon}: |P^{*}_{a,\gamma}f_{\nu}(x)|\gtrsim \nu^{(a-1)-\frac{a}{\gamma}}\}|\\
\lesssim& \nu^{-2((a-1)-\frac{a}{\gamma})}\|f_{\nu}\|^{2}_{H^{s}}\\
\lesssim& \nu^{2(\frac{a}{\gamma})}\nu^{a-4s-\frac{a}{\gamma}}.
\end{split}
\end{equation}
If we take $\epsilon =\frac14$, then
$$
E_{\frac{1}{4}} > m[0,1] - \tfrac{1}{4}\geq \tfrac12m[0,1].
$$
For $\forall \lambda\in (0, \frac{1}{10})$, we claim that there exists $x_{0}\in[0,1]$ such that
\begin{equation}\label{translation}
m\big({E_{\epsilon}\cap[x_{0}, x_{0}+\lambda]}\big) \ge \tfrac{1}{2}m[x_{0},x_{0}+\lambda].
\end{equation}
Indeed, we split the interval $[0,1]$ as $[0,1] = \bigcup_{k}[x_{k},x_{k}+\lambda]$. Then we have
$$
m{E_{\epsilon}} = \sum_{k}m\big(E_{\epsilon}\cap[x_{k}, x_{k}+\lambda]\big),
$$
and
$$
m[0,1] = \sum_{k}m[x_{k}, x_{k}+\lambda].
$$
By pigeonholing, we prove the claim \eqref{translation}.

Let $\lambda = a\nu^{\frac{2a}{\gamma}-2(a-1)}$ and $\tilde{f}_{\nu}(x) = f_{\nu}(x-x_{0})$, then we have by \eqref{weak(2,2)}
\begin{align*}
\frac{1}{2}a\nu^{\frac{2a}{\gamma}-2(a-1)} &= \frac{1}{2}|[x_{0}, x_{0}+a\nu^{\frac{2a}{\gamma}-2(a-1)}]|
\le |E_{\epsilon}\cap[x_{0}, x_{0}+a\nu^{\frac{2a}{\gamma}-2(a-1)}]|\\
&\le |\{x\in E_{\epsilon}\cap[x_{0}, x_{0}+a\nu^{\frac{2a}{\gamma}-2(a-1)}]: |P^{*}_{a,\gamma}\tilde{f}_{\nu}(x)|\gtrsim \nu^{(a-1)-\frac{a}{\gamma}}\}|\\
&\le C_{\epsilon}\nu^{-2((a-1)-\frac{a}{\gamma})}\|\tilde{f}_{\nu}\|^{2}_{H^{s}} = C_{\epsilon}\nu^{-2((a-1)-\frac{a}{\gamma})}\|f_{\nu}\|^{2}_{H^{s}}\\
&\le C_{\epsilon}\nu^{2(\frac{a}{\gamma})}\nu^{a-4s-\frac{a}{\gamma}}.
\end{align*}
This inequality yields
\begin{equation}\label{controdict}
\nu^{a-4s-\frac{a}{\gamma}} \gtrsim a
\end{equation}
Let $\nu$ tend to $0$, then $\nu^{a-4s-\frac{a}{\gamma}} \rightarrow 0$ for $s < \frac{1}{4}a(1-\frac{1}{\gamma})$, which contradicts with \eqref{controdict}.

In conclusion, for $0<a<1, $if $s < \frac{1}{4}a\big(1-\frac{1}{\gamma}\big)$, we see that the weak type $(2,2)$ inequality
$$
|\{x\in\mathbb{R}; |P^{*}_{a,\gamma}f_{\nu}(x)|\gtrsim \nu^{(a-1)-\frac{a}{\gamma}}\}|
\lesssim \nu^{-2\big((a-1)-\frac{a}{\gamma}\big)} \|f_{\nu}\|^{2}_{H^{s}}
$$
fails, which implies almost everywhere convergence fails either.

Next, we turn to look at the case $a=1$.

{\bf Case 2: $a=1.$}
 Similarly by contradiction, we assume that
 the almost everywhere convergence \eqref{aeconv} holds for $s<\frac{1}{2}\big(1-\frac{1}{\gamma}\big)$ with $\gamma>1$.
Then, using Proposition \ref{prop:poconwe}, we obtain for any $ \epsilon >0$, there exists a set $\tilde{E}_{\epsilon} \subset [0,1]$ with $m\tilde{E}_{\epsilon}> 1-\epsilon$ such that
\begin{equation}\label{equ:weak2-21-1}
\big|\{x\in \tilde{E}_{\epsilon}: |P^{*}_{1,\gamma}f(x)| > \lambda\}\big|
\le C_{\epsilon}\lambda^{-2}\|f\|^{2}_{H^{s}},
\end{equation}
for $\forall \lambda >0$, $f\in H^{s}(\mathbb{R})$ and $s<\frac{1}{2}\big(1-\frac{1}{\gamma}\big)$. As before, we will construct an counterexample to get an contradiction.

For $N\ge1$ and $\gamma>1$, we choose the set $A=[-N,-\frac{N}{2}]$ and $E=[0,N^{-\frac{1}{\gamma}}]$. Let
$$
\hat{f}_{A}(\xi) = \chi_{A}(\xi),
$$
and we have
\begin{equation}\label{norm}
\|f_{A}\|_{H^{s}} = \left(\int \left(1+|\xi|^{2}\right)^{s}|\hat{f}_{A}|^{2}\,\mathrm{d}\xi\right)^{\frac{1}{2}} \sim N^{s}N^{\frac{1}{2}}
\end{equation}
It is easy to see that
\begin{equation}\label{eg1}
|P_{1,\gamma}^{t}f_{A}|
= \left|\int_{A} e^{ix\xi}e^{it|\xi|}e^{-t^{\gamma}|\xi|}\,\mathrm{d}\xi\right|.
\end{equation}
Choose $t=x\in E$, then $|x^{\gamma}N|\le1$ and
\begin{equation}\label{c1}
\sup_{0<t<1}|P_{1,\gamma}^{t}f_{A}|
\ge\left|\int_{A} e^{ix\xi}e^{it|\xi|}e^{-t^{\gamma}|\xi|}\,\mathrm{d}\xi\right|
= \left|\int_{A} e^{-t^{\gamma}|\xi|}\,\mathrm{d}\xi\right| \gtrsim Ne^{-x^{\gamma}N} \gtrsim N.
\end{equation}
By \eqref{translation}, there exists $\tilde{x}_{0}\in[0,1]$ such that
\begin{equation}\label{translation-1}
m(\tilde{E}_{\epsilon}\cap[\tilde{x}_{0},\tilde{x}_{0}+N^{-\frac1\gamma}]) \ge \tfrac{1}{2}m[\tilde{x}_{0},\tilde{x}_{0}+N^{-\frac1\gamma}].
\end{equation}

Let $\tilde{f}_{A}(x) = f_{A}(x-\tilde{x}_{0})$, then by \eqref{equ:weak2-21-1}, \eqref{c1} and \eqref{translation-1}, there holds that
\begin{align*}
\frac{1}{2}N^{-\frac1\gamma} &= \frac{1}{2}|[x_{0}, x_{0}+N^{-\frac1\gamma}]|
\le |\tilde{E}_{\epsilon}\cap[x_{0}, x_{0}+N^{-\frac1\gamma}]|\\
&\le |\{x\in \tilde{E}_{\epsilon}\cap[x_{0}, x_{0}+N^{-\frac1\gamma}]: |P^{*}_{1,\gamma}\tilde{f}_{A}(x)|\gtrsim N\}|\\
&\le C_{\epsilon}N^{-2}\|\tilde{f}_{A}\|^{2}_{H^{s}} = C_{\epsilon}N^{-2}\|f_{A}\|^{2}_{H^{s}}\\
&\le C_{\epsilon}N^{2s-1}.
\end{align*}
From this inequality, we can obtain
\begin{equation}\label{c3}
N^{2s-1+\frac1\gamma} \ge C.
\end{equation}
Let $N$ tend to $\infty$, then $N^{2s-1+\frac1\gamma} \rightarrow \infty$ for $s < \frac{1}{2}(1-\frac{1}{\gamma})$, which contradicts with \eqref{c3}.

In conclusion, for $a=1$, if $s < \frac{1}{2}(1-\frac{1}{\gamma})$, we see that the weak type $(2,2)$ inequality
$$
|\{x\in\mathbb{R}; |P^{*}_{1,\gamma}f_{A}(x)|\gtrsim N\}|
\lesssim N^{-2} \|f_{A}\|^{2}_{H^{s}}
$$
fails, which implies almost everywhere convergence fails either.

\end{proof}

\section{Hausdorff Dimension of Divergent points}

In this section, we will discuss the problem for the set of the divergent points. First we need to establish the maximal estimate for the operator $P^{t}_{a,\gamma}$ with $0<a\leq1$ and $\gamma>0$ in the general Borel measure $\mu \in \mathfrak{M}(B(0,1))$. As a consequence, we  obtain Theorem \ref{main-thm-dim-1} by the Frostman lemma below.

\vskip-0.8cm

\begin{proof}[{\bf The proof of Theorem $\ref{lemma:main-a}$}]

Using the Kolmogrov-Selierstov-Plessner method, one can find a Borel function $t = t(x): \mathbb{R} \rightarrow (0,1)$, and a Borel function $\omega\in L^{\infty}(\mu)$ with $\|\omega\|_{L^{\infty}(\mu)} \le 1$ such that
\begin{align*}
\|P^{\ast}_{a,\gamma}f(x)\|_{L^{1}(\mu)}
 \le &~ 2\|P^{t(x)}_{a,\gamma} f(x)\|_{L^{1}(\mu)} \le ~2\liminf_{N \rightarrow \infty} \|P^{t(x)}_{a,\gamma,N}f(x)\|_{L^{1}(\mu)} \\
\le &~ 2 \liminf_{N \rightarrow \infty} \int P^{t(x)}_{a,\gamma,N}f(x)~\overline{\omega(x)}\,\mathrm{d}\mu (x).
\end{align*}
Take the function $\chi(\xi)\in S(\mathbb{R})$ such that
$$
\chi(\mathbb{R}) \subset [0,1], \,\,\chi(\xi) = 1 \,\,\text{on} \,\,[-1,1], \,\,{\rm supp}\,\, \chi(\xi)\subset[-2,2].
$$
Then by Fubini's theorem, Cauchy-Schwarz's inequality, and Lemma \ref{maximal estimate}, similar to \eqref{equ:kxyest}, we obtain
\begin{align*}
&\left|\int(P^{t(x)}_{a,\gamma,N}f(x))\right.\left.\vphantom{\int_{B(0,1)}}\overline{\omega(x)}\,\mathrm{d}\mu (x)\right|^{2} \\
=~ & \left|\int_{\mathbb{R}}\hat{f}(\xi) 
\eta(\frac{\xi}{N})
\int e^{it(x)|\xi|^{a}}e^{-t(x)^{\gamma}|\xi|^{a}}e^{ix\xi}g(x)\eta(\frac{x}{N})\,\mathrm{d}\mu(x)\mathrm{d}\xi \right|^{2}\\
\le ~ &\|f\|^{2}_{H^{s}(\mathbb{R})} \iint |\omega(x)\overline{\omega(y)}|(1+~W(x,y)~)\mathrm{d}\mu(x)\mathrm{d}\mu(y),
\end{align*}
where
\begin{equation}
\begin{split}
W(x,y)= ~ &\left|\int_{\mathbb{R}}e^{i(t(x)-t(y))|\xi|^{a}-i(y-x)\xi}(1+\xi^{2})^{-s}(1-\chi(\xi))e^{-(t(x)^{\gamma}+t(y)^{\gamma})|\xi|^{a}}\eta(\frac{\xi}{N})^{2}
\,\mathrm{d}\xi\right|.\\
\end{split}
\end{equation}
Then by Lemma
\ref{maximal estimate} , Remark \ref{maximal estimate-a>1} and  the fact that $\|\omega\|_{L^{\infty}(\mu)} \le 1$, we finish the proof.

\end{proof}

To prove  Theorem \ref{main-thm-dim-1},
we recall the following lemma, which builds the relation between $\textup{dim}\,\,U$ and the energy $I_{s}(\mu)$ of Borel measue $\mu$ on $U$. 

\begin{lemma}[Frostman, \cite{Mattila}]\label{dim}
For a Borel set $U\subset \mathbb{R}^{n}$,
\begin{center}
  \textup{dim} U = $\sup$\{s: there is $\mu\in\mathfrak{M}(U)$ such that $I_{s}(\mu)<\infty$\}.
\end{center}
\end{lemma}

\begin{proof}[{\bf The proof of Theorem $\ref{main-thm-dim-1}$}]



First, we give the detail proof of Theorem \ref{main-thm-dim-1}(i).
Suppose $I_{1-2s}(\mu)<\infty$. Let
$$
\mathcal {Q} = \big\{x\in \mathbb{R}: P^{t}_{a,\gamma}f(x) \not\rightarrow f(x)\ as\ t\rightarrow 0 \big\}.
$$
$\mathcal {Q}$ is obviously a Borel set. Let $\lambda\in(0,1)$, $\epsilon>0$ and choose a smooth function $g$ for which $\|f-g\|_{H^{s}(\mathbb{R})}<\lambda\epsilon$. Since $\lim_{t\rightarrow 0^+}P^{t}_{a,\gamma}g(x)=g(x)$, for any $ x\in \mathbb{R}$, then we have
\begin{equation}\label{last-1}
\begin{split}
\limsup_{t\rightarrow 0}|P^{t}_{a,\gamma}f(x)-f(x)|
\leq~  P^{\ast}_{a,\gamma}(f-g)(x)+|f(x)-g(x)|,
\end{split}
\end{equation}
for any \ $x\in \mathbb{R}$.
Therefore, by the estimate \eqref{energy-est} and the Cauchy-Schwarz inequality, we have
\begin{equation}\label{last-2}
\begin{split}
&\mu(\{x: \limsup_{t\rightarrow 0}|P^{t}_{a,\gamma}f(x)-f(x)|>\lambda\})\\
\lesssim & \lambda^{-1}\sqrt{I_{1-2s}(\mu)}\|f-g\|_{H^{s}(\mathbb{R})} +\lambda^{-1} \|f-g\|_{L^1(d\mu)}\\
\lesssim &\sqrt{I_{1-2s}(\mu)}\epsilon,
\end{split}
\end{equation}
which yields $\mu(\mathcal {Q})=0$.

If we assume $\textup{dim~} \mathcal {Q}> 1-2s$, then by Lemma \ref{dim}, there exists a Borel measure $\mu\in\mathfrak{M}(\mathcal {Q})$ with $I_{1-2s}(\mu)<\infty$.
And by \eqref{last-1} and \eqref{last-2}, we have $\mu(\mathcal {Q})=0$, which contradicts with the condition $\mu\in\mathfrak{M}(\mathcal {Q})$ and $0<\mu(\mathcal {Q})<\infty$.
In conclusion, we have $\textup{dim~} \mathcal {Q}\leq 1-2s$ in this case.

The proofs of Theorem \ref{main-thm-dim-1}(ii)(iii)(iv) are similar to that of (i) above.
Therefore, we conclude the  proof of Theorem \ref{main-thm-dim-1}.

\end{proof}


\begin{center}

\end{center}

\end{document}